\documentclass{amsart}

 \usepackage[all]{xy}
\usepackage{lmodern}
\usepackage[dvipsnames,svgnames,x11names,hyperref]{xcolor}
\usepackage{mathtools,url,graphicx,verbatim,amssymb,enumerate,stmaryrd,nicefrac}
\usepackage[pagebackref,colorlinks,citecolor=Mahogany,linkcolor=Mahogany,urlcolor=Mahogany,filecolor=Mahogany]{hyperref}
\usepackage{microtype}
\usepackage[margin=1.5in]{geometry}
\usepackage{setspace}
\usepackage{tikz, tikz-cd}
\usetikzlibrary{matrix,calc,arrows,patterns,decorations.pathreplacing}
\usepackage{caption}

\newtheorem*{corollary*}{Corollary}

\theoremstyle{definition}

\title{Realizing the Tutte polynomial as a cut-and-paste K-theoretic invariant} 

\author{Mauricio Gomez Lopez}
\email{gomezlom@lafayette.edu}
\address{Department of Mathematical Sciences \\
	       Lafayette College, Pardee Hall\\
	       Easton, PA, 18042 \\USA}

\date{\today}

\begin{document}

\maketitle
\vspace{-0.5cm}
\begin{abstract}
Cut-and-paste $K$-theory is a new variant of higher algebraic $K$-theory that has proven to be useful in problems involving decompositions of combinatorial and geometric objects, e.g., scissors congruence of polyhedra and reconstruction problems in graph theory. In this paper, we show that this novel machinery can also be used in the study of matroids. Specifically, via the $K$-theory of categories with covering families developed by Bohmann-Gerhardt-Malkiewich-Merling-Zakharevich, we realize the Tutte polynomial map of Brylawski (also known as the \textit{universal Tutte-Grothendieck invariant} for matroids) as the $K_0$-homomorphism induced by a map of $K$-theory spectra.
\end{abstract}

\tableofcontents 

\section{Introduction} \label{intro}

The work in this paper rests upon the following two mathematical foundations: 

\vspace{0.2cm}

\begin{itemize}

\item[$\cdot$] The work of Brylawski on decompositions of matroids in \cite{Br71}. We remark that matroids  are referred to as 
\textit{combinatorial pregeometries} in \cite{Br71}.  

\vspace{0.2cm}

\item[$\cdot$] The emerging field of \textit{cut-and-paste K-theory,} a new variant of higher algebraic $K$-theory suited 
to the study of decompositions of geometric and combinatorial objects 
(e.g., decompositions of polyhedra \cite{Zak12}, \cite{Zak17}, \cite{BGMMZ}). 

\end{itemize}

\vspace{0.2cm}

Cut-and-paste $K$-theory (also known as \textit{combinatorial K-theory}) has its roots in two different, yet related, research areas: The
work of Zakharevich in \cite{Zak12} and \cite{Zak17} to study scissors congruence problems through the lens of algebraic $K$-theory, 
and the study of the $K$-theory of varieties carried out 
by Campbell and Zakharevich in \cite{Zak17b}, \cite{Cam19}, and \cite{CZ22}. 
There are currently several forms of cut-and-paste $K$-theory available in the literature. These include Zakharevich's  \textit{$K$-theory of assemblers}
\cite{Zak17}; \textit{Squares $K$-theory,} developed in 
\cite{CKMZ23}, which offers the appropriate framework to deal with
cut-and-paste problems involving manifolds (see \cite{HMM}); and
\textit{$K$-theory of categories with covering families,} recently introduced in \cite{BGMMZ}, which 
generalizes Zakharevich's $K$-theory of assemblers. 

It is this last form of combinatorial $K$-theory that we shall work with
in this paper. A great deal of the power of this construction 
is due to the generality of 
the definition of category with covering families, 
which allows much 
flexibility in terms of applications. As explained in \cite{BGMMZ}, 
there are many natural examples of categories with covering families, 
including Grothendieck  sites, assemblers, categories of polyhedra, and groups,
all endowed with a suitable \textit{covering family structure} (a notion we shall introduce 
in Definition \ref{catcov}). Another recent entry to this list of applications is the 
work of Calle and Gould in \cite{CG}, in which they give a $K$-theoretic
formulation of the edge reconstruction conjecture for graphs using the
machinery of categories with covering families from \cite{BGMMZ}. 

This paper offers a new application of $K$-theory of categories with 
covering families. Namely, we use this $K$-theoretic machinery to study
invariants of matroids, specifically the \textit{Tutte polynomial,} arguably the
most fundamental matroid invariant. Matroids are combinatorial objects that 
are meant to simultaneously generalize finite configurations of vectors and graphs. 
Consider, for example, the following two figures: 

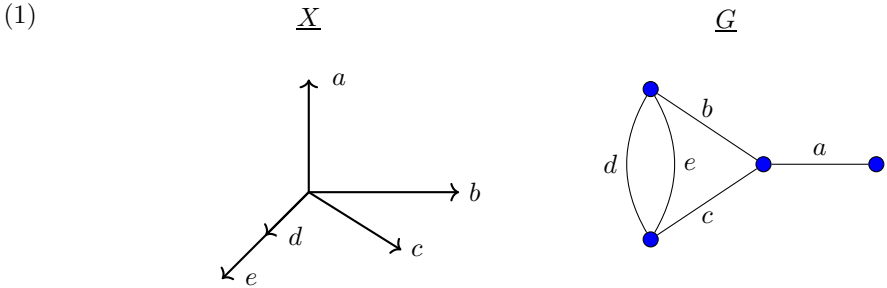
\begin{equation} \label{ardila.matroids}
\quad
\end{equation}
\vspace{-1cm}
\[
\quad \textcolor{white}{a} \hspace{2.3cm}
\begin{tikzpicture}
[axis/.style={->,thick}]
\node at (0,2.3,0) {\underline{$X$}};
	\draw[axis] (0,0,0) -- (2,0,0) node[anchor=west]{$b$};
        \draw[axis] (0,0,0) -- (2,0,2) node[anchor=west]{$c$};
	\draw[axis] (0,0,0) -- (0,1.5,0) [axis/.style={->,red}] node[anchor=west]{\textcolor{white}{a}$a$};
	\draw[axis] (0,0,0) -- (0,0,1.5) node[anchor=west] {\textcolor{white}{a}$d$};
	\draw[axis] (0,0,0) -- (0,0,3) node[anchor=west] {\textcolor{white}{a}$e$};
	
\end{tikzpicture}
\qquad \qquad
\begin{tikzpicture}
      \tikzset{enclosed/.style={draw, circle, inner sep=0pt, minimum size=.2cm, fill=blue}}
      
      \node (v0) at (-5, 0.9) {\underline{$G$}};
      \node[enclosed] (v1) at (-6,0) {};
      \node[enclosed] (v2) at (-6,-2) {};
      \node[enclosed] (v3) at (-4.5, -1) {};
      \node[enclosed] (v4) at (-3, -1) {};
      \node at (0,-2.6) {};
      
      \draw [-] (v2) to [out = 120, in=240] node[left,midway] {$d$} (v1);
      \draw [-] (v2) to [out = 60, in=300] node[right,midway] {$e$} (v1);
      \draw (v1) -- (v3) node[midway, above] {$b$};
      \draw (v2) -- (v3) node[midway, below] {$c$};
      \draw (v3) -- (v4) node[midway, above] {$a$};
\end{tikzpicture}
\]
Both the configuration $X$ of vectors in $\mathbb{R}^3$ on the left and the graph $G$ on the right induce a matroid (in fact, they induce isomorphic matroids).
More precisely, as we shall discuss in \S \ref{matroid.theory.background}, a matroid consists of a finite \textit{ground set} $E$ and 
a collection of subsets of $E$, called \textit{independent sets}. For example, referring to (\ref{ardila.matroids}), 
the ground set for the matroids induced by $X$ and $G$ is  
$E = \{ a, b, c, d, e \}$. For the matroid induced by $X$, the independent sets
are the sets of vectors which are linearly independent. On the other hand, for the matroid induced by
$G$, the independent sets are the sets of edges that do not contain any closed edge-paths (see Example \ref{minor.ex}). 
As the reader can verify, both matroids have the same independent sets. As we will see in 
Definition \ref{matroid}, the independent sets of a matroid must satisfy certain axioms, which 
are meant to capture the combinatorial essence of linear independence.  

Given a matroid $M$, its \textit{Tutte polynomial} $T(M; x,y)$ 
is a polynomial in two variables $x,y$ with positive integer coefficients. This polynomial,
which can be regarded as a generalization of the chromatic and flow polynomials for graphs, 
has its roots in the work of Tutte in \cite{Tut47}, which anticipated the use
of $K$-theoretical constructions in the study of graphs. The name \textit{Tutte polynomial} was
penned by Crapo in \cite{Cra69}.  Further landmark contributions to the study
of the Tutte polynomial were made by Brylawski in \cite{Br71}, in which he carried over to matroid theory 
the $K$-theoretic ideas permeating commutative algebra and algebraic topology at the time.  

It is precisely the work of Brylawski in \cite{Br71} that motivates and facilitates the connection between cut-and-paste
$K$-theory and matroid theory that we explore in this paper. To formulate our main results, let us
denote by $\mathcal{M}$ the set of isomorphism classes of matroids. The notion of matroid isomorphism, plus other
background material from matroid theory, will be reviewed in \S \ref{matroid.theory.background}. 
Given that the Tutte polynomial is an isomorphism invariant, we can define a function 
$\mathcal{M} \rightarrow \mathbb{Z}[x,y]$ which sends an isomorphism class $[M]$ to its
Tutte polynomial $T(M; x,y)$. This map naturally extends to a group homomorphism 
\begin{equation} \label{tutte.homom} 
\mathcal{T}: \mathbb{Z}[\mathcal{M}] \longrightarrow \mathbb{Z}[x,y]
\end{equation} 
defined on the free abelian group $\mathbb{Z}[\mathcal{M}]$. 
Our first main theorem states that this group homomorphism 
$\mathcal{T}$ can be realized as the $0$-th level of a map between 
two $K$-theory spectra. To make sense of this statement, it is helpful
to give a brief overview of the pipeline of $K$-theory of categories with covering families: 

\begin{itemize}

\item[$\bullet$] This $K$-theory takes as input a category with a \textit{covering family structure}, i.e., 
a small category $\mathcal{C}$ together with a collection $\mathcal{S}$ of finite multi-sets 
$\{ C_i \rightarrow B\}_{i\in I}$ of morphisms in $\mathcal{C}$, called \textit{covering families}, subject to certain conditions.  
We shall elaborate this definition in \S \ref{k.theory.background}. 

\vspace{0.2cm}

\item[$\bullet$] The output of this $K$-theory pipeline is a spectrum $K(\mathcal{C})$. 
An important feature of this spectrum is that its $0$-th homotopy group $K_0(\mathcal{C})$ records 
all possible ways of decomposing objects in $\mathcal{C}$ via the covering families in $\mathcal{S}$. This statement
shall be made precise in Theorem \ref{k0.result}.   

\end{itemize}

In Section \S \ref{matroid.theory.background}, we shall define a (small) category 
$\mathbf{Mat}_{+}$ of matroids, with a distinguished base-point object $\ast$,
where morphisms
are functions preserving independent sets
(more details will be provided in Definitions 
\ref{catmat} and \ref{catmat.base}). 
Then, in Section \S \ref{cov.fam.sec}, we shall construct two 
covering family structures on $\mathbf{Mat}_{+}$: 

\begin{itemize} 

\vspace{0.2cm}

\item[(i)] $\mathcal{S}^{\cong}$, consisting of all possible isomorphisms in $\mathbf{Mat}_{+}$.

\vspace{0.2cm}

\item[(ii)] $\mathcal{S}^{\mathrm{tc}}$, consisting of a class of covering families called \textit{Tutte coverings}. 
One can view these coverings as categorical reformulations of the \textit{Tutte decompositions} considered by
Brylawski in \cite{Br71}.
\end{itemize}

The triples $(\mathbf{Mat}_{+}, \mathcal{S}^{\cong}, \ast)$ and 
 $(\mathbf{Mat}_{+}, \mathcal{S}^{\mathrm{tc}}, \ast)$ are then 
categories with covering families, in the sense of \cite{BGMMZ},  
which we shall denote by $\mathbf{Mat}^{\cong}$ and  $\mathbf{Mat}^{\mathrm{tc}}$,
respectively (we will revisit these constructions in detail in 
 \S \ref{cov.fam.sec}). 
In light of Theorem \ref{k0.result}, we will immediately have
\begin{equation} \label{k0.iso.mat}
K_0(\mathbf{Mat}^{\cong}) = \mathbb{Z}[\mathcal{M}].
\end{equation}
Also, by construction, $\mathcal{S}^{\mathrm{tc}}$ shall
extend the structure $\mathcal{S}^{\cong}$, which will give us 
a canonical morphism 
\begin{equation} \label{first.gamma}
\Gamma: \mathbf{Mat}^{\cong} \longrightarrow \mathbf{Mat}^{\mathrm{tc}}
\end{equation}
of categories with covering families (see Remark \ref{catcov.cat}), which in turn 
induces a group homomorphism 
\begin{equation} \label{second.gamma}
\gamma: K_0(\mathbf{Mat}^{\cong}) \longrightarrow K_0(\mathbf{Mat}^{\mathrm{tc}}).
\end{equation}
Our first main theorem in this paper describes the isomorphism type of 
the group $K_0(\mathbf{Mat}^{\mathrm{tc}})$ and provides an explicit description
of the homomorphism $\gamma$ appearing in (\ref{second.gamma}). More precisely,
we have the following.  

\theoremstyle{plain} \newtheorem{thma}{Theorem}[section]
\renewcommand{\thethma}{\Alph{thma}}

\begin{thma} \label{thma}
For the category with covering families $\mathbf{Mat}^{\mathrm{tc}}$, the following holds:

\begin{itemize}

\vspace{0.2cm}

\item[(i)] There is a canonical isomorphism 
\begin{equation} \label{k0.tc.mat}
\rho: K_0(\mathbf{Mat}^{\mathrm{tc}}) \stackrel{\cong}{\longrightarrow} \mathbb{Z}[x,y]
\end{equation}
of abelian groups. 

\vspace{0.2cm}

\item[(ii)] Via the identifications 
$K_0(\mathbf{Mat}^{\cong}) = \mathbb{Z}[\mathcal{M}]$
and $K_0(\mathbf{Mat}^{\mathrm{tc}}) \cong \mathbb{Z}[x,y]$
given in (\ref{k0.iso.mat}) and (\ref{k0.tc.mat}) respectively, 
the homomorphism
$\gamma: K_0(\mathbf{Mat}^{\cong}) \rightarrow K_0(\mathbf{Mat}^{\mathrm{tc}})$
induced by the morphism 
$\Gamma: \mathbf{Mat}^{\cong} \rightarrow \mathbf{Mat}^{\mathrm{tc}}$
is the group homomorphism 
\[
\mathcal{T}: \mathbb{Z}[\mathcal{M}] \longrightarrow \mathbb{Z}[x,y]
\]
which maps an isomorphism class $[M]$ to its Tutte polynomial $T(M; x,y)$.
\end{itemize}

\end{thma}

Part (ii) of the previous statement gives the realization 
of the \textit{Tutte polynomial map}
$\mathcal{T}: \mathbb{Z}[\mathcal{M}] \longrightarrow \mathbb{Z}[x,y]$
as the $0$-th level of a map between $K$-theory spectra that
we promised earlier in this introduction.  

As we shall see in Section 
\S \ref{ring.struct.sec}, the direct sum operation on matroids
induces a ring structure on both 
$K_0(\mathbf{Mat}^{\cong})$ and $K_0(\mathbf{Mat}^{\mathrm{tc}})$.
With these ring structures in place, the group homomorphisms 
appearing in Theorem \ref{thma} become ring homomorphisms. 
Explicitly, we shall prove the following.

 \theoremstyle{plain} \newtheorem{thmb}[thma]{Theorem}

\begin{thmb} \label{thmb}
Let $+$ denote the addition operation in both 
$K_0(\mathbf{Mat}^{\cong})$ and $K_0(\mathbf{Mat}^{\mathrm{tc}})$. 

\begin{itemize}

\vspace{0.2cm}

\item[(i)] In both $K_0(\mathbf{Mat}^{\cong})$ and $K_0(\mathbf{Mat}^{\mathrm{tc}})$, 
setting
\begin{equation} \label{prod.oplus}
[M] \cdot [N] := [M \oplus N]
\end{equation}
gives a well-defined product on generators, and hence a product on
$K_0(\mathbf{Mat}^{\cong})$ and $K_0(\mathbf{Mat}^{\mathrm{tc}})$. 
Furthermore, the operations $+$, $\cdot$ define
commutative ring structures on 
$K_0(\mathbf{Mat}^{\cong})$ and $K_0(\mathbf{Mat}^{\mathrm{tc}})$.

\vspace{0.2cm}

\item[(ii)] With the ring structures defined above, the group homomorphisms 
$\rho$, $\gamma$, and $\mathcal{T}$ from Theorem \ref{thma} become ring homomorphisms. 

\end{itemize}

\end{thmb}

As we shall discuss after the proof of Theorem \ref{thmb}, the ring $K_0(\mathbf{Mat}^{\mathrm{tc}})$
agrees with the \textit{Tutte-Grothendieck ring} $\mathcal{R}_{\mathrm{TG}}$ constructed by 
Brylawski in \cite{Br71}. This ring $\mathcal{R}_{\mathrm{TG}}$ is a free commutative ring
with two generators: one corresponding to an \textit{isthmus} $\varepsilon$ and the other
corresponding to a \textit{loop} $\sigma$. Also, the map
$\gamma: K_0(\mathbf{Mat}^{\cong}) \rightarrow K_0(\mathbf{Mat}^{\mathrm{tc}})$, 
which we can also write as
\[
\gamma: \mathbb{Z}[\mathcal{M}] \longrightarrow \mathcal{R}_{\mathrm{TG}},
\]
is what Brylawski referred to as the \textit{Tutte polynomial} in \cite{Br71}. 
Adopting the terminology used in \cite{GMc}, we shall call the map 
$\gamma: \mathbb{Z}[\mathcal{M}] \rightarrow \mathcal{R}_{\mathrm{TG}}$
the \textit{universal Tutte-Grothendieck invariant} (the sense in which $\gamma$ is
universal shall become clear in the discussion following the proof of Theorem \ref{thmb}).  
It is a consequence of Theorem \ref{thma} that the universal Tutte-Grothendieck invariant 
$\gamma$ lifts (as a morphism of abelian groups) to a map of spectra. More concretely, 
we can rephrase the statement of Theorem \ref{thma} as follows.

 \theoremstyle{plain} \newtheorem{thmc}[thma]{Theorem}

\begin{thmc} \label{thmc}
The map of $K$-theory spectra 
\[
K(\Gamma): K(\mathbf{Mat}^{\cong}) \rightarrow  K(\mathbf{Mat}^{\mathrm{tc}})
\]
induced by the morphism $\Gamma: \mathbf{Mat}^{\cong} \rightarrow \mathbf{Mat}^{\mathrm{tc}}$ 
is a lift of the universal Tutte-Grothendieck invariant 
$\gamma: \mathbb{Z}[\mathcal{M}] \rightarrow \mathcal{R}_{\mathrm{TG}}$
(as a morphism of abelian groups) to the category of spectra. 
\end{thmc}

This article is structured as follows: Sections \S \ref{k.theory.background} and
\S \ref{matroid.theory.background} are meant to provide background; 
in \S \ref{k.theory.background}, we review the key notions and results from 
$K$-theory of categories with covering families, whereas in 
\S \ref{matroid.theory.background} we give a thorough review of several
fundamental concepts and constructions from matroid theory. In particular, 
we give a self-contained discussion of the Tutte polynomial. While there
are other ways of introducing this polynomial (e.g., via the corank-nullity polynomial),
we shall present this invariant using its standard recursive definition  
(see \cite{Br71} and \S 9 of \cite{GMc}).
The heart of this paper is Section \S \ref{cov.fam.sec}. In this section, we construct 
the covering family structures $\mathcal{S}^{\cong}$ and $\mathcal{S}^{\mathrm{tc}}$
on the category $\mathbf{Mat}_{+}$, and present the proofs of our main theorems.
We close Section \S \ref{cov.fam.sec} by briefly discussing the possibility 
of finding a spectrum-level lift for the Tutte polynomial map $\mathcal{T}$
as a ring homomorphism and not just as a map of abelian groups (see Note \ref{klift.rem}).

We remark that some of the proofs in Section \S \ref{cov.fam.sec} 
(e.g., the proofs of Propositions \ref{ind.refine} and \ref{ind.ind}) are similar in structure to some 
of the arguments presented in \cite{Br71}.
Nevertheless, besides facilitating a connection between matroid theory 
and the modern ideas of cut-and-paste $K$-theory, we believe that the 
categorical nature of our constructions make these proofs more
streamlined and structured.  \\

\noindent \textbf{Acknowledgements.} The author is grateful to Gary Gordon 
for bringing to his attention the work by Brylawski on the Tutte polynomial. The author
also thanks the Department of Mathematical Sciences at Lafayette College 
for offering a collegial and supportive environment during the development of this project. 

\section{K-theory of categories with covering families} \label{k.theory.background}

The purpose of this section is to give an overview of the main
ideas from \textit{cut-and-paste K-theory} we shall use throughout this paper. 
We will start by reviewing the notion of \textit{category with covering families}
presented in  
\cite{BGMMZ} (see also Section \S 2 of \cite{CG} for a helpful discussion of this construction).   

\vspace{0.2cm} 

\theoremstyle{definition} \newtheorem{catcov}{Definition}[section]

\begin{catcov} \label{catcov}
Let $\mathcal{C}$ be a small category. 

\begin{itemize}

\item[(a)] A \textit{multi-morphism} is a finite (possibly empty) multi-set of morphisms in $\mathcal{C}$ of the form
\[
\{ f_i: C_i \rightarrow B\}_{i \in I}.
\]
More explicitly, a \textit{multi-morphism} is a finite collection (possibly with repetitions) of morphisms in $\mathcal{C}$
 with a common target $B$. 

\vspace{0.2cm}

\item[(b)] Now, suppose that $\mathcal{C}$ has a distinguished 
\textit{base-point} object with the property that

\vspace{0.15cm}

\begin{center}
$\mathcal{C}(\ast, \ast) = \{ \mathrm{Id}_{\ast} \}$ and 
$\mathcal{C}(C,\ast) = \varnothing$ whenever $C \neq \ast$.
\end{center}

\vspace{0.15cm} 

\noindent A \textit{covering family structure} on $\mathcal{C}$ is a collection $\mathcal{S}$
of multi-morphisms with the following properties: 

\begin{itemize}

\vspace{0.15cm}

\item[(i)] For every finite (possibly empty) set $I$, the
family $\{ \ast \rightarrow \ast\}_{i\in I}$ is in $\mathcal{S}$.  

\vspace{0.15cm}

\item[(ii)] For every object $C \in \mathcal{C}$, the singleton 
$\{ \mathrm{Id}_C: C \rightarrow C\}$ is in $\mathcal{S}$. 

\vspace{0.15cm}

\item[(iii)]  Let $J = \{ 1, \ldots, n\}$ and
suppose  $\{ g_j: B_j \rightarrow C\}_{j \in J}$ 
is a multi-morphism in $\mathcal{S}$. Then, given 
a collection of multi-morphisms in $\mathcal{S}$ of the form
\[
\{ f_{i1}: A_{i1} \rightarrow B_1\}_{i\in I_1} \quad \ldots\ldots \quad \{ f_{in}: A_{in} \rightarrow B_n\}_{i\in I_n},
\]
the collection of compositions
\[
\bigcup_{j\in J}\{ g_j \circ f_{ij}: A_{ij} \rightarrow C\}_{j\in J, \hspace{0.1cm} i \in I_j}
\]
is also a multi-morphism in $\mathcal{S}$.   
\end{itemize}

\vspace{0.15cm}

\noindent The multi-morphisms in $\mathcal{S}$ are called \textit{covering families.}
If $\mathcal{C}$ is a small category, $\ast$ a base-point object of $\mathcal{C}$,  and $\mathcal{S}$ 
a covering family structure on $\mathcal{C}$,  
then the triple $(\mathcal{C}, \mathcal{S}, \ast)$ shall be called a \textit{category with covering families.} 
Moreover, given an object $B\in \mathcal{C}$, we shall often call a multi-morphism 
$\{ f_i: C_i \rightarrow B\}_{i \in I}$ in $\mathcal{S}$ 
a \textit{covering of $B$}. 
\end{itemize}

\end{catcov}

In the above definition, it is useful to view a multi-morphism 
$\{ f_i: C_i \rightarrow B\}_{i \in \{ 1, \ldots, n\}}$ in the covering structure 
$\mathcal{S}$ as a rule for decomposing the object $B$ into 
smaller pieces $C_1$, \ldots, $C_n$. Each such multi-morphism 
gives a different way of decomposing the object $B$.

\vspace{0.2cm}

\theoremstyle{definition} \newtheorem{catcov.cat}[catcov]{Remark}

\begin{catcov.cat} \label{catcov.cat}
We can define a category $\mathbf{CatFam}$ of categories with covering families by declaring a morphism  
\[
(\mathcal{C}_1, \mathcal{S}_1, \ast_1) \rightarrow (\mathcal{C}_2, \mathcal{S}_2, \ast_2)
\]
to be a functor $\mathcal{F}: \mathcal{C}_1 \rightarrow \mathcal{C}_2$ that 
preserves the required structure, i.e., $\mathcal{F}(\ast_1) = \ast_2$ and $\mathcal{F}$
must map covering families in $\mathcal{S}_1$ to covering families in $\mathcal{S}_2$.   
\end{catcov.cat}

\theoremstyle{definition} \newtheorem{base.point}[catcov]{Remark}

\begin{base.point} \label{base.point}
As explained in \cite{BGMMZ}, it might be possible for a category $\mathcal{C}$
to have a family of multi-morphisms satisfying the conditions (ii) and (iii) described above and yet not have a base-point object (in the sense of the previous definition).  
In this case, we can add a \textit{disjoint base-point.} In other words, from $\mathcal{C}$,
we form a new category $\mathcal{C}_+$ by adding an object $\ast$ satisfying 
\begin{center}
$\mathrm{Hom}_{\mathcal{C}_+}(\ast, \ast) = \{ \mathrm{Id}_{\ast} \}$ and 
$\mathrm{Hom}_{\mathcal{C}_+}(C,\ast) = \mathrm{Hom}_{\mathcal{C}_+}(\ast, C) = \varnothing$ for all $C \in \mathcal{C}$.  
\end{center}
\end{base.point}

As mentioned in the introduction, the $K$-theory of categories with covering families developed in 
\cite{BGMMZ} is a generalization of Zakharevich's $K$-theory of 
assemblers \cite{Zak17}. As also explained in the introduction, 
this kind of $K$-theory takes as input 
a category with covering families $\mathcal{C}$, and produces a 
spectrum $K(\mathcal{C})$. We shall give a brief overview of this
construction shortly. 
Before doing so, it is worth 
recalling that the group $\pi_0$ of $K(\mathcal{C})$,
conventionally denoted by $K_0(\mathcal{C})$, records the 
different ways we can decompose objects in $\mathcal{C}$
via the multi-morphisms in the covering family structure $\mathcal{S}$.
More concretely, we have the following result, presented as Proposition 3.8 in 
\cite{BGMMZ}. The proof of this result is analogous to that 
of Theorem 2.13 in \cite{Zak17}.  

\vspace{0.2cm}

\theoremstyle{plain} \newtheorem{k0.result}[catcov]{Theorem} 

\begin{k0.result} \label{k0.result}
If $(\mathcal{C}, \mathcal{S}, \ast)$ is a category with covering families, then the group 
$K_0(\mathcal{C})$ is the free abelian group 
$\mathbb{Z}[\mathrm{Ob}(\mathcal{C})]$ modulo the relations 
$[A] = \sum_{j\in J} [A_j]$ for any covering family   
 $\{ A_j \rightarrow A \}_{j \in J}$ in $\mathcal{S}$.  
\end{k0.result}

\vspace{0.2cm}

\theoremstyle{definition} \newtheorem{k0.result.remark}[catcov]{Remark} 

\begin{k0.result.remark} \label{k0.result.remark}
Note that, for any category with covering families 
$(\mathcal{C}, \mathcal{S}, \ast)$, condition (i)
in the definition of covering family structure (part (b) of Definition \ref{catcov})
forces the class $[\ast]$ to be the 
identity element in the group $K_0(\mathcal{C})$. 
\end{k0.result.remark}

The $K$-theory
construction for a category with covering families
presented in \cite{BGMMZ}
relies on the notions introduced in
the next definition and Definition \ref{cov.pointed} below. 

\theoremstyle{definition} \newtheorem{cat.of.cov}[catcov]{Definition} 

\begin{cat.of.cov} \label{cat.of.cov}
Given a category with covering families $(\mathcal{C}, \mathcal{S}, \ast)$, we define 
its \textit{category of covers} $W(\mathcal{C})$ to be the category whose 
objects are finite multi-sets of objects $\{ A_i \}_{i \in I}$ in $\mathcal{C}$, and a morphism
\[
\{ B_j \}_{j\in J} \longrightarrow \{ A_i \}_{i \in I}
\]
between two objects $\{ B_j \}_{j\in J}$ and $ \{ A_i \}_{i \in I}$ consists of the following data:

\vspace{0.2cm}

\begin{itemize}

\item[$\cdot$] A set function $f: J \rightarrow I$.

\vspace{0.2cm}

\item[$\cdot$] For each $i \in I$, a covering family 
$\{ g_{ij}: B_j \rightarrow A_i \}_{j \in f^{-1}(i)}$ belonging to $\mathcal{S}$.   
\end{itemize}

We compose two morphisms in $W(\mathcal{C})$ by first composing the underlying set functions
and then composing covering families using condition (iii) of part (b) of Definition \ref{catcov}. 

\end{cat.of.cov}

As explained in \cite{BGMMZ}, 
for any category with covering families 
$(\mathcal{C}, \mathcal{S}, \ast)$, 
its category of covers $W(\mathcal{C})$
has a natural base-point object, corresponding to $I = \varnothing$. 
Moreover, as proven in \cite{BGMMZ}, the construction described in Definition \ref{cat.of.cov} above
defines a functor 
\[
W(-): \mathbf{CatFam} \rightarrow \mathbf{Cat}_{\mathrm{pt}}
\]
from the category of categories with covering families to the category of categories with base-points. 

The other ingredient for our desired $K$-theory construction is given in Definition \ref{cov.pointed} below.
Before stating this definition, we need one technical preliminary: Given a pointed set $X$ with base-point $\ast$, we can view 
$X$ as a pointed category by taking the set of objects to be $X$ itself and by defining a
morphism set $\mathrm{Hom}(a,b)$ to be a one-point set if $a = b$ or $a = \ast$. Otherwise, $\mathrm{Hom}(a,b)$ is empty. 

\theoremstyle{definition} \newtheorem{cov.pointed}[catcov]{Definition} 

\begin{cov.pointed} \label{cov.pointed}
Consider a category with covering families $(\mathcal{C}, \mathcal{S}, \ast)$. For simplicity, 
we shall denote this category by $\mathcal{C}$.  
Then, if $X$ is a pointed set,  $X\wedge \mathcal{C}$ is the category with
covering families
whose set of objects is given by 
\[
\mathrm{ob}(X\wedge \mathcal{C}) = \big(\mathrm{ob}\hspace{0.02cm}X\times \mathrm{ob}\hspace{0.05cm}\mathcal{C}\big) / \big(\mathrm{ob}\hspace{0.02cm}X\vee \mathrm{ob}\hspace{0.05cm}\mathcal{C}\big)
\]
and whose morphisms are induced by those in $X \times \mathcal{C}$. 
The base-point object is the wedge-point of $\mathrm{ob}(X\wedge \mathcal{C})$. Intuitively, we can
imagine $X\wedge \mathcal{C}$ as the category obtained by 
taking several copies of $\mathcal{C}$ (one copy per element in $X$), and gluing all of them
at their base-points. As explained in \cite{BGMMZ}, the covering family structure 
on $X\wedge \mathcal{C}$ is given by declaring a multi-set $\{A_i \rightarrow B\}_{i\in I}$
to be a covering family if the objects $A_i$ and $B$  are contained in a single copy of $\mathcal{C}$
and $\{A_i \rightarrow B\}_{i\in I}$ is a multi-set belonging to $\mathcal{S}$.
\end{cov.pointed}

\vspace{0.2cm}

\theoremstyle{definition} \newtheorem{K-cons}[catcov]{Note}

\begin{K-cons} \label{K-cons}
\textbf{(The $K$-theory of a category with covering families)} Consider a category with covering families 
$(\mathcal{C}, \mathcal{S}, \ast)$, and let $S^1_{\bullet}$ denote the simplicial circle. As explained in 
Definition 2.17 of \cite{BGMMZ}, the assignment
\[
X \mapsto |N_{\bullet}W(X\wedge \mathcal{C})|
\]
defines a functor from pointed sets to pointed spaces. In fact,
this functor is a $\Gamma$-space (see \cite{Seg74}). Then, as defined in \cite{BGMMZ}, 
the \textit{$K$-theory spectrum} $K(\mathcal{C})$ of $(\mathcal{C}, \mathcal{S}, \ast)$
is the symmetric spectrum associated to this $\Gamma$-space. More concretely, 
the $k$-th level of $K(\mathcal{C})$ is the realization of the simplicial set
\[
p \mapsto N_pW(S^k_p\wedge \mathcal{C}),
\]
where $S^k_{\bullet} = (S^1_{\bullet})^{\wedge k}$ is the simplicial $k$-sphere. 
The details of the construction of the structure maps of this spectrum can be found in 
Definition 2.12 of \cite{Zak17}. It is worth pointing out that the construction from \cite{Zak17} is formulated in the specific context of assemblers. However,
the procedure presented in \cite{Zak17} carries over without difficulties to the more general
case of categories with covering families. 
\end{K-cons}

Besides the preliminaries we have already discussed in this section, we shall also use the following notion in our proof
of Theorem \ref{thma}. This definition is inspired by the
notion of \textit{indecomposable object} from \cite{Br71}.  

\vspace{0.2cm}

\theoremstyle{definition} \newtheorem{ind}[catcov]{Definition}

\begin{ind} \label{ind}
Let $(\mathcal{C}, \mathcal{S}, \ast)$ be a category with covering families. We shall say 
that an object $B \in \mathcal{C}$ is \textit{indecomposable} if 
the only coverings of $B$ are singletons of the form 
$\{ C \stackrel{\cong}{\longrightarrow} B\}$, i.e., 
the only covering families in $\mathcal{S}$ with target $B$ are singletons 
with a single isomorphism mapping to $B$. 
\end{ind} 

\section{Matroid theory essentials}    \label{matroid.theory.background}

\subsection{Basic definitions} 
In this section, we will
collect the main definitions and facts
from matroid theory that we will need for the constructions 
we will discuss later in this paper. There are several equivalent (or, in the language  
of matroid theory, \textit{cryptomorphic}) ways
of defining a matroid. In this paper, we shall mainly use the following
definition, which is perhaps the most standard way of defining a matroid.   

\theoremstyle{definition} \newtheorem{matroid}{Definition}[section]

\begin{matroid} \label{matroid}
A \textit{matroid} $M$ is a tuple $(E, \mathcal{I})$ consisting 
of the following data: (1) A finite set $E$, and (2) a collection
$\mathcal{I}$ of subsets of $E$ satisfying the following axioms: 

\begin{itemize}

\vspace{0.15cm}

\item[(I1)] The empty set $\varnothing$ is in $\mathcal{I}$.

\vspace{0.15cm}

\item[(I2)] If $I \in \mathcal{I}$ and $J \subseteq I$, then $J \in \mathcal{I}$. 

\vspace{0.15cm}

\item[(I3)] \textit{Augmentation axiom.} If $I, J \in \mathcal{I}$ and 
$|J| < |I|$, then there exists an element $x \in I - J$ such that the set
$J\cup \{x\}$  also belongs to $\mathcal{I}$. 

\end{itemize}

The set $E$ is called the \textit{ground set} of the matroid $M$, 
and the subsets of the collection $\mathcal{I}$ are called the 
\textit{independent sets} of $M$. 
\end{matroid}

The following is a list of matroid-theoretic notions we will need for the rest of this paper.

\theoremstyle{definition} \newtheorem{list.mat}[matroid]{Essential definitions}

\begin{list.mat} \label{list.mat}

Fix a matroid $M = (E, \mathcal{I})$. 

\begin{itemize}

\vspace{0.15cm}

\item[(a)] A subset $B$ of $E$ is a \textit{basis} of $M$ if it is a maximal independent set. 
In other words, $B$ is a basis if $B \in \mathcal{I}$ and there
is no other independent set containing $B$. It is a standard fact of
matroid theory that any two bases must have the
same cardinality. 

\vspace{0.15cm}

\item[(b)] Any subset $X$ which does not belong to $\mathcal{I}$ is called
a \textit{dependent set.} In particular, minimal dependent sets are called 
\textit{circuits}. 

\vspace{0.15cm}

\item[(c)] An element $x \in E$ is said to be a \textit{loop} if the singleton 
$\{x\}$ is a circuit. Moreover, two distinct elements $x, y \in E$ are said 
to be \textit{parallel} if the set $\{x,y\}$ is a circuit.

\end{itemize}

\end{list.mat}

As mentioned earlier, there are several equivalent ways of
defining a matroid. For example, one can define a matroid 
by simply specifying its collection of bases $\mathcal{B}$. 
This way of describing matroids is convenient for 
defining the following matroid operation.  

\theoremstyle{definition} \newtheorem{dual.mat}[matroid]{Definition}

\begin{dual.mat} \label{dual.mat}
Suppose $M = (E, \mathcal{B})$ is a matroid, where $\mathcal{B}$
is its collection of bases. The \textit{dual of $M$} is the matroid 
$M^* = (E, \mathcal{B}^*)$ whose collection of bases is 
$\mathcal{B}^* = \{ E- B \hspace{0.1cm}| \hspace{0.1cm} B \in \mathcal{B} \}$.  

\end{dual.mat}

It is a standard fact that $M^*$ is also a matroid. 
At this point, it is convenient to 
introduce a few more basic notions from matroid theory. 

\theoremstyle{definition} \newtheorem{list.mat.2}[matroid]{Essential definitions (continued)}

\begin{list.mat.2} \label{list.mat.2}

Fix a matroid $M = (E, \mathcal{I})$, and let 
$\mathcal{B}$ be its collection of bases.  

\begin{itemize}

\vspace{0.15cm}

\item[(d)] An element $e \in E$ is an \textit{isthmus} of $M$ if $e$ is contained
in every basis $B \in \mathcal{B}$. Equivalently, $e$ is an isthmus 
of $M$ if and only if $e$ is a loop of $M^*$. 

\vspace{0.15cm}

\item[(e)] Two elements $e, f \in E$ are said to be \textit{coparallel}
if $\{ e, f \}$ is a circuit in $M^*$, i.e., $e$ and $f$ are parallel in $M^*$.  

\end{itemize}

\end{list.mat.2}

Elements that are neither isthmuses nor loops shall be important in 
many of our later arguments. For this reason, it is convenient to have a
special name for such elements. 

\theoremstyle{definition} \newtheorem{non.deg}[matroid]{Definition}

\begin{non.deg} \label{non.deg}
We shall say that an element $e$ of a matroid $M$ is \textit{non-degenerate} 
if it is neither an isthmus nor a loop. 
\end{non.deg}

\subsection{Matroid operations} Taking duals (Definition \ref{dual.mat}) is
one operation we can perform on matroids. 
The next definition gives two more examples of operations which 
produce new matroids from old ones. 

\theoremstyle{definition} \newtheorem{del.con}[matroid]{Definition}

\begin{del.con} \label{del.con}
Fix a matroid $M = (E, \mathcal{I})$, where $\mathcal{I}$ is its 
collection of independent sets. 

\begin{itemize}

\vspace{0.15cm}

\item[(i)] \textit{Deletion.} Suppose $e \in E$ is not an isthmus of $M$. We define
$M \backslash e$ to be the matroid with ground set $E - \{ e\}$ and whose
collection of independent sets is defined as  
\[
\mathcal{I}_{M \backslash e} := \big\{ I \subseteq E - \{e\} \hspace{0.1cm} | \hspace{0.1cm} I \in \mathcal{I} \big\}.
\]
We say that $M \backslash e$ is the matroid
obtained from $M$ by \textit{deleting} $e$. 

\vspace{0.15cm}

\item[(ii)] \textit{Contraction.} On the other hand, if $e$ is not a loop of $M$, we define 
$M / e$ to be the matroid with ground set $E - \{ e\}$ and whose
collection of independent sets is defined as  
\[
\mathcal{I}_{M / e} := \big\{ I \subseteq E - \{e\} \hspace{0.1cm} | \hspace{0.1cm} I\cup\{e\} \in \mathcal{I} \big\}.
\]
In this case, we say that $M / e$ is the matroid 
obtained from $M$ by \textit{contracting} $e$. 

\end{itemize}

\end{del.con}

\vspace{0.2cm}

The next proposition describes how the 
deletion and contraction operations interact with duality. For this statement,
we need to introduce the following terminology: 
We say that two
matroids $M_1 = (E_1, \mathcal{I}_1)$ and $M_2 = (E_2, \mathcal{I}_2)$
are \textit{isomorphic}, written as $M_1 \cong M_2$, if there is a
bijection $f: E_1 \rightarrow E_2$ with the property that $I \in \mathcal{I}_1$
if and only if $f(I) \in \mathcal{I}_2$.  

\theoremstyle{plain} \newtheorem{dual.del.con}[matroid]{Proposition}

\begin{dual.del.con} \label{dual.del.con}

Fix a matroid $M = (E, \mathcal{I})$. If $e$ is 
a non-degenerate element 
of $M$ (in the sense of Definition \ref{non.deg}), 
then we have matroid isomorphisms
of the form 
\[
(M/ e)^* \cong M^* \backslash e \qquad (M \backslash e)^* \cong M^* / e.
\]
The function of sets underlying both of these isomorphisms is the identity map on $E - \{e\}$. 

\end{dual.del.con}

In other words, deletion and contraction are dual operations. 
It also turns out that these two operations commute with each other
and with themselves, 
as the following proposition indicates.

\theoremstyle{plain} \newtheorem{del.con.comm}[matroid]{Proposition}

\begin{del.con.comm} \label{del.con.comm}
Fix a matroid $M = (E, \mathcal{I})$, and fix two elements $e$ and $f$ of $M$. 

\begin{itemize}

\vspace{0.2cm}

\item[(i)] If $e$ and $f$ are not coparallel and are not isthmuses of $M$, then 
\[
\big(M \backslash e\big) \backslash f = \big(M \backslash f \big) \backslash e.
\]

\vspace{0.2cm}

\item[(ii)] If $e$ and $f$ are not parallel and are not loops of $M$, then 
\[
\big(M / e \big)/ f = \big(M / f\big) / e.
\]

\vspace{0.2cm}

\item[(iii)] If $e$ is not an isthmus and $f$ is not a loop of $M$, then 
\[
\big(M \backslash e \big)/ f = \big(M /f \big) \backslash e.
\] 

\end{itemize}

\end{del.con.comm}

\vspace{0.2cm}

According to this proposition, it does not matter in which order we perform deletions and contractions, as long
as the element we wish to delete (resp. contract) is not an isthmus (resp. a loop). We shall typically drop
parentheses when denoting matroids obtained by multiple deletions and contractions. So, for example,
we will write $\big(M \backslash e \big)/ f$ simply as $M \backslash e / f$.
Proofs for both Propositions \ref{dual.del.con} and \ref{del.con.comm} can be found 
in standard matroid theory references, such as \cite{GMc} and \cite{Ox}.  

Matroids obtained from other matroids via an iteration of deletions and contractions receive the following name in the literature.   

\theoremstyle{definition} \newtheorem{minor}[matroid]{Definition}

\begin{minor} \label{minor}
Fix a matroid $M = (E, \mathcal{I})$. Any matroid obtained from $M$ via a sequence
of deletions and/or contractions is called a \textit{minor of $M$}. 
\end{minor}

\theoremstyle{definition} \newtheorem{minor.ex}[matroid]{Example}

\begin{minor.ex} \label{minor.ex}
Any finite graph $G$ induces naturally a matroid: If $E_G$ is the set
of edges of $G$, then we can define a matroid 
$M_G = (E_G, \mathcal{I}_G)$ by taking $\mathcal{I}_G$ 
to be all subsets of edges that do not form any closed edge-paths.
A matroid induced by a graph in this way is called a
\textit{graphical matroid}. 
For example, take the following graph $G$: 

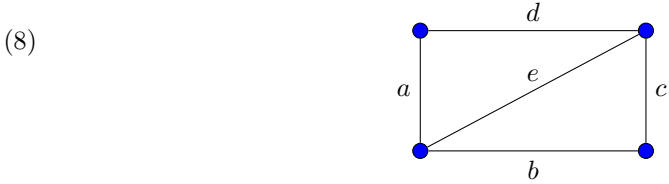
\begin{equation} \label{pic.graph} 
\quad
\end{equation}
\vspace{-1.4cm}
\[
\begin{tikzpicture}
      \tikzset{enclosed/.style={draw, circle, inner sep=0pt, minimum size=.2cm, fill=blue}}

      \node[enclosed] (E) at (0.75,2.35) {};
	   \node[enclosed] (A) at (3.75,2.35) {};
	    \node[enclosed] (B) at (3.75,0.75) {};
      \node[enclosed] (L) at (0.75,0.75) {};
      
      \draw (E) -- (L) node[midway, left] (edge1)  {$a$};
      \draw (L) -- (B) node[midway, below] (edge2) {$b$};
      \draw (B) -- (A) node[midway, right] (edge3) {$c$};
      \draw (A) -- (E) node[midway, above] (edge4) {$d$};
      \draw (L) -- (A) node[midway, above] (edge5) {$e$};
\end{tikzpicture}
\]
Then, the bases of the matroid $M_G$ induced by this graph are 
\[
\{a, b, c\} \quad \{a, b, d\} \quad \{ a,c,d\} \quad \{b, c, d\} \quad \{a, b, e\} \quad \{a, c, e\} \quad \{ e,b,d\} \quad \{e, c, d\}.
\]
On the other hand, the circuits of $M_G$ are $\{ a, d, e\}$, $\{b, c, e\}$, and $\{a,b,c,d\}$. 
Any subset of $E_G$ not containing any of these three subsets is independent. 
Performing
deletion and contraction on a graphical matroid corresponds to 
deleting and contracting edges in the underlying graph. So, for example, 
if $G_1$ and $G_2$ are the graphs  
obtained by contracting $e$ and deleting $a$ respectively 
in $G$ (see figure (\ref{pic.subgraphs})), then we have $M_{G_1} = M_G/e$ and $M_{G_2} = M_G \backslash a.$

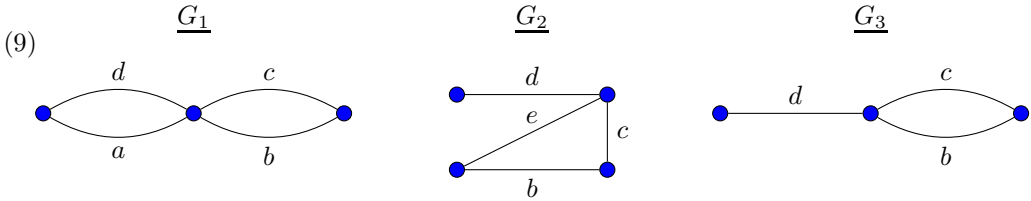
\begin{equation} \label{pic.subgraphs}
\quad
\end{equation}
\vspace{-1.4cm}
\[
\begin{tikzpicture}
      \tikzset{enclosed/.style={draw, circle, inner sep=0pt, minimum size=.2cm, fill=blue}}
      
      \node (A0) at (2, 2.8) {\underline{$G_1$}};
       \node (B0) at (6.5,2.8) {\underline{$G_2$}};
         \node (C0) at (11,2.8) {\underline{$G_3$}};
      \node[enclosed] (E) at (0,1.55) {};
	   \node[enclosed] (A) at (2,1.55) {};
	    \node[enclosed] (B) at (4,1.55) {};    
	    \node[enclosed] (E1) at (5.5,1.8) {};
	   \node[enclosed] (A1) at (7.5,1.8) {};
	    \node[enclosed] (B1) at (7.5,0.8) {};
      \node[enclosed] (L1) at (5.5,0.8) {};
         \node[enclosed] (E2) at (9,1.55) {};
	   \node[enclosed] (A2) at (11,1.55) {};
	    \node[enclosed] (B2) at (13,1.55) {}; 
	    
       \draw [-] (E) to [out=30,in=150] node[above]  {$d$} (A);
       \draw [-] (E) to [out=-30,in= -150] node[below]  {$a$} (A);
       \draw [-] (A) to [out=30,in=150] node[above]  {$c$} (B);
       \draw [-] (A) to [out=-30,in= -150] node[below]  {$b$} (B);
       \draw (L1) -- (B1) node[midway, below] (edge2) {$b$};
      \draw (B1) -- (A1) node[midway, right] (edge3) {$c$};
      \draw (A1) -- (E1) node[midway, above] (edge4) {$d$};
      \draw (L1) -- (A1) node[midway, above] (edge5) {$e$};
      \draw [-] (E2) to node[above]  {$d$} (A2);
       \draw [-] (A2) to [out=30,in=150] node[above]  {$c$} (B2);
       \draw [-] (A2) to [out=-30,in= -150] node[below]  {$b$} (B2);
\end{tikzpicture}
\]
The graph $G_3$ on the far-right is obtained by contracting
$e$ and deleting $a$ in $G$. For this graph, we have 
$M_{G_3} = M_{G}/e\backslash a$
(equivalently, $M_{G_3} = M_G \backslash a /e$).  
\end{minor.ex}

All the operations we have discussed so far 
require only one single matroid as input. 
We will close this subsection by giving an example
of an operation that takes multiple matroids
as input in order  
to generate a new matroid.

\theoremstyle{definition} \newtheorem{direct.sum}[matroid]{Definition}

\begin{direct.sum} \label{direct.sum}
Fix two matroids $M_1 = (E_1, \mathcal{I}_1)$ 
and $M_2 = (E_2, \mathcal{I}_2)$. The 
\textit{direct sum of $M_1$ and $M_2$,} denoted by 
$M_1 \oplus M_2$, is the matroid whose ground set $E$
and collection of independent sets $\mathcal{I}$ are defined
respectively as follows: 

\begin{itemize}

\item[$\cdot$] $E = E_1 \sqcup E_2$ (i.e., $E$ is the disjoint union of the ground sets $E_1$ and $E_2$). 

\vspace{0.15cm}

\item[$\cdot$] $\mathcal{I} = \big\{ I_1\sqcup I_2 \hspace{0.1cm} | \hspace{0.1cm} I_1 \in \mathcal{I}_1, \hspace{0.1cm}  I_2 \in \mathcal{I}_2\big\}$.

\end{itemize}
 
\end{direct.sum}

\vspace{0.2cm}

Direct sums of more than two matroids are defined inductively. Also,
it is evident that $M_1\oplus M_2 \cong M_2 \oplus M_1$. 
If $e$ is not an isthmus of $M_1$ and $f$ is not a loop of $M_2$, 
then it is straightforward to verify the following identities:  
\begin{equation} \label{direct.sum.del.con}
(M_1 \oplus M_2)\backslash e = (M_1 \backslash e) \oplus M_2 \qquad \qquad (M_1 \oplus M_2)/ f = M_1 \oplus (M_2 / f).
\end{equation}

\theoremstyle{definition} \newtheorem{ist.loop}[matroid]{Remark}

\begin{ist.loop} \label{ist.loop}
Let $E$ be a set consisting of a single element $e$. Then, there are only two matroids we can define 
on $E = \{ e \}$: One by declaring $e$ to be an isthmus, and the other one by declaring $e$ to be a loop. 
From now on, we will denote these two matroids by $\varepsilon$ and $\sigma$, i.e.,
\begin{equation} \label{ist.loo.not}
\varepsilon = (E, \{ e\}) \qquad \sigma = (E, \varnothing). 
\end{equation}
We shall typically denote the $n$-fold direct sum of $\varepsilon$ (resp. $\sigma$) 
with itself by $\varepsilon^n$ (resp. $\sigma^n$). 
If $M$ is a matroid with no non-degenerate elements, then it is evident that
\[
M \cong \varepsilon^m \oplus \sigma^n,
\]
where $m$ and $n$ are the number of isthmuses and loops in $M$ respectively. 
\end{ist.loop}

\vspace{0.2cm} 

\subsection{Categories of matroids}

Multiple definitions of a \textit{category of matroids} are already 
available in the literature (see for example \cite{HP} and \cite{Ig}). For the 
purposes of this paper, we define this category as follows. 

\theoremstyle{definition} \newtheorem{catmat}[matroid]{Definition}

\begin{catmat} \label{catmat}

Let $\mathbf{Mat}$ denote the category consisting of the following data:   

\begin{itemize}

\vspace{0.15cm}

\item[$\cdot$] Objects of $\mathbf{Mat}$ are matroids 
$M = (E, \mathcal{I})$ such that $E \subset \{ 1, 2, \ldots \}$.

\vspace{0.15cm}

 \item[$\cdot$] A morphism $M \rightarrow N$ from $M = (E_1, \mathcal{I}_1)$
 to $N = (E_2, \mathcal{I}_2)$ is an injective
 set function $f:E_1 \rightarrow E_2$ with the property
that $f(I) \in \mathcal{I}_2$ for any $I \in \mathcal{I}_1$. 

\end{itemize}

\vspace{0.15cm}

If $M' = (E', \mathcal{I}')$ is a minor of 
$M = (E, \mathcal{I})$ (in which case, $E'$ is a subset of $E$), 
then the morphism $M' \rightarrow M$ induced by the obvious
inclusion of sets $E' \hookrightarrow E$ shall be called 
\textit{the standard inclusion of $M'$ into $M$.}  

\end{catmat}

\vspace{0.2cm}

The condition $E \subset \{ 1, 2, \ldots \}$ imposed on objects
guarantees that $\mathbf{Mat}$ is a small category. Recall that
any category with covering families (in the sense of Definition
\ref{catcov}) is required to have a distinguished base-point
object. A natural choice for such a base-point in $\mathbf{Mat}$
would seem to be \textit{the empty matroid}, i.e., the matroid
on the empty set $\varnothing$ whose unique independent set
is $\varnothing$ itself. By abuse of notation, we shall denote
the empty matroid simply by $\varnothing$. However, 
such a choice of base-point would be undesirable because,
as we shall indicate in Definition \ref{tutte.pol},
the Tutte polynomial of $\varnothing$ is $T(\varnothing; x, y) = 1$.
On the other hand, according to Remark \ref{k0.result.remark},
the matroid we take as the base-point should correspond to the
idenity element in the $K_0$ group.
For this reason, we are required to add a base-point to 
$\mathbf{Mat}$, as indicated in the next definition. 

\theoremstyle{definition} \newtheorem{catmat.base}[matroid]{Definition}

\begin{catmat.base} \label{catmat.base}
We define $\mathbf{Mat}_{+}$ to be the category obtained by
adding a disjoint base-point object $\ast$ to $\mathbf{Mat}$, in the sense of
Remark \ref{base.point}.
Furthermore, we can extend the 
direct sum operation to $\mathbf{Mat}_{+}$ by declaring 
\[
M\oplus \ast = \ast \oplus M = \ast \hspace{0.1cm}\text{ for all objects $M$ in }\mathbf{Mat}_{+}.
\] 
\end{catmat.base}

\vspace{0.2cm}

Finally, for some of the arguments we will present in the
next section, it is convenient to work with 
\textit{multi-sets of matroids.} By a \textit{multi-set of matroids}
we shall mean a collection of matroids $\{ M_i \}_{i \in \Lambda}$
indexed by some finite non-empty set $\Lambda$. Note that
it is possible to have $M_i = M_j$ even if $i$ and $j$ are distinct indices in $\Lambda$.
We shall also extend the notion of matroid isomorphism to these more
general kinds of objects. 

\theoremstyle{definition} \newtheorem{iso.multi}[matroid]{Definition}

\begin{iso.multi} \label{iso.multi}
An \textit{isomorphism of multi-sets of matroids} 
$\{ M_i\}_{i \in \Lambda} \stackrel{\cong}{\longrightarrow} \{ N_j\}_{j \in \Omega}$
consists of the following data: 

\begin{itemize}

\vspace{0.15cm}

\item[(i)] A bijection $g: \Lambda \stackrel{\cong}{\longrightarrow} \Omega$, and

\vspace{0.15cm}

\item[(ii)] for each $i \in \Lambda$, an isomorphism of matroids
$f_i: M_i   \stackrel{\cong}{\longrightarrow} N_{g(i)}$.

\end{itemize}

\end{iso.multi}

\vspace{0.2cm}

\subsection{The Tutte polynomial} 

Our next goal is to discuss the main matroid invariant we shall focus on 
throughout the rest of this paper: \textit{the Tutte polynomial}. 
As mentioned in the introduction, 
there are multiple ways of
defining this invariant. However, in this paper, we shall opt for the following 
recursive definition of the Tutte polynomial, since this is the definition that
motivates the categorical constructions we will develop in the next section 
(see also Definition 9.2 in \cite{GMc}).

\theoremstyle{definition} \newtheorem{tutte.pol}[matroid]{Definition}

\begin{tutte.pol} \label{tutte.pol}
The \textit{Tutte polynomial} $T(M; x, y)$ of a matroid $M$ is defined 
recursively as follows: 

\begin{enumerate} 

\item $T(M; x, y) = T(M \backslash e; x, y) + T(M / e; x, y)$ if $e$ is a non-degenerate element of $M$.

\vspace{0.2cm}

\item $T(M; x, y) = x\cdot T(M / e; x, y)$ if $e$ is an isthmus. 

\vspace{0.2cm}

\item $T(M; x, y) = y\cdot T(M \backslash e; x, y)$  if $e$ is a loop.

\vspace{0.2cm}

\item $T(M; x , y) = 1$ if $M = \varnothing$.  

\end{enumerate}

\end{tutte.pol}

\vspace{0.2cm}

\theoremstyle{definition} \newtheorem{tutte.pol.ex0}[matroid]{Example}

\begin{tutte.pol.ex0} \label{tutte.pol.ex0}
Consider a matroid of the form $\varepsilon^m \oplus \sigma^n$, i.e., 
a matroid with $m$ isthmuses, $n$ loops, and no non-degenerate elements (see Remark \ref{ist.loop}). 
It follows from rules (2)-(4) in the previous definition that 
\begin{equation} \label{basic.tutte}
T(\varepsilon^m \oplus \sigma^n; x ,y) = x^my^n. 
\end{equation}
\end{tutte.pol.ex0}

\vspace{0.2cm} 

At this point, 
besides the identity given in (\ref{basic.tutte}),
it is also helpful to give an example of a computation of a non-trivial Tutte polynomial
in complete detail.  
This example will not only help the reader process the previous definition,
but it will also motivate the construction of the covering family structure $\mathcal{S}^{\mathrm{tc}}$
on the category of matroids $\mathbf{Mat}_{+}$ that we will use
in the proof of Theorem \ref{thma}.  

\theoremstyle{definition} \newtheorem{tutte.pol.ex}[matroid]{Example}

\begin{tutte.pol.ex} \label{tutte.pol.ex}
Consider again the graph $G$ from Example \ref{minor.ex}:
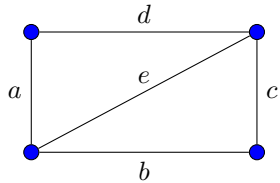
\[
\begin{tikzpicture}
      \tikzset{enclosed/.style={draw, circle, inner sep=0pt, minimum size=.2cm, fill=blue}}

      \node[enclosed] (E) at (0.75,2.35) {};
	   \node[enclosed] (A) at (3.75,2.35) {};
	    \node[enclosed] (B) at (3.75,0.75) {};
      \node[enclosed] (L) at (0.75,0.75) {};
      
      \draw (E) -- (L) node[midway, left] (edge1)  {$a$};
      \draw (L) -- (B) node[midway, below] (edge2) {$b$};
      \draw (B) -- (A) node[midway, right] (edge3) {$c$};
      \draw (A) -- (E) node[midway, above] (edge4) {$d$};
      \draw (L) -- (A) node[midway, above] (edge5) {$e$};
\end{tikzpicture}
\]
We will compute the Tutte polynomial of the matroid induced by $G$
(which, for simplicity, we will denote by $M$) 
using only the first rule listed in Definition \ref{tutte.pol}
and the identity given in (\ref{basic.tutte}). 
To perform this computation, we will reduce 
 $M$ into matroids of the form $\varepsilon^m \oplus \sigma^n$
 by repeatedly applying deletions and contractions. 
 We illustrate this process in the following diagram:
 
 \begin{equation} \label{first.tree}
\quad 
\end{equation}
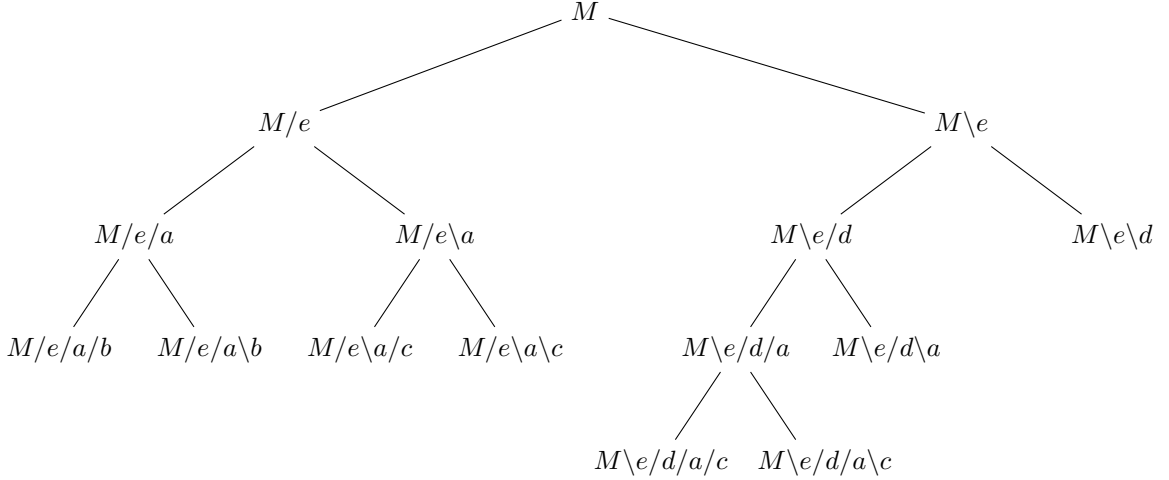
\[
\begin{tikzpicture}
      \tikzset{enclosed/.style={draw, circle, inner sep=0pt, minimum size=.2cm, fill=blue}}

      \node (M) at (0,0) {$M$};
      \node (Me0) at (-4,-1.5) {$M/e$};
         \node (Ma0) at (-6,-3) {$M/e/a$};
         \node (Mb0) at (-7, -4.5) {$M/e/a/b$};
         \node (Mb1) at (-5, -4.5) {$M/e/a\backslash b$};
      \node (Ma1) at (-2,-3) {$M/e\backslash a$};
      \node (Mc0) at (-3,-4.5) {$M/e\backslash a/c$};
       \node (Mc1) at (-1,-4.5) {$M/e\backslash a\backslash c$};
      \node (Me1) at (5,-1.5) {$M\backslash e$};
     \node (Md0) at (3,-3) {$M\backslash e/d$};
     \node (Ma2) at (2, -4.5) {$M\backslash e/d/a$};
     \node (Mc2) at (1, -6)  {$M\backslash e/d/a/c$};
     \node (Mc3) at (3, -6)  {\quad$M\backslash e/d/a\backslash c$}; 
     \node (Ma3) at (4, -4.5) {$M\backslash e/d\backslash a$};
      \node (Md1) at (7,-3) {$M\backslash e\backslash d$};
	
	\draw (M) -- (Me0);
	\draw (M) -- (Me1);
	\draw (Me0) -- (Ma0);
	\draw (Me0) -- (Ma1);
	\draw (Me1) -- (Md0);
	\draw (Me1) -- (Md1);
	\draw (Mb0) -- (Ma0);
	\draw (Mb1) -- (Ma0);
	\draw (Mc0) -- (Ma1);
	\draw (Mc1) -- (Ma1);
	\draw (Ma2) -- (Md0);
	\draw (Ma3) -- (Md0);
	\draw (Mc2) -- (Ma2);
	\draw (Mc3) -- (Ma2);
	
\end{tikzpicture}
\]
Going from top-to-bottom, this tree is obtained by 
taking a matroid located at a `node' and then 
contracting and deleting a non-degenerate element
from that matroid. For example, at the very top,
we contract and delete the edge $e$ in $M$. Then, at 
the node $M/e$, we obtain the next two minors
by contracting and deleting the edge $a$. 
By the first rule listed in Definition \ref{tutte.pol},
the Tutte polynomial of a matroid located at a node
is the sum of the Tutte polynomials of the two matroids 
corresponding to the two `branches' sprouting
downwards from the node. For example, for the matroid
$M/e$ located at the left-hand node in the second upper-most
level, we have
\[
T(M/e; x, y) = T(M/e/a; x, y) + T(M/e \backslash a; x, y). 
\]
As the reader
can verify, all 
the `leaves' of the tree (\ref{first.tree}) are labeled by minors which are 
isomorphic to matroids of the form
$\varepsilon^m \oplus \sigma^n$.  
More precisely, going from left-to-right, the matroids at the 
leaves have the following isomorphism types: 
\[
\sigma^2 \qquad\quad \varepsilon\oplus\sigma \qquad\quad \varepsilon\oplus\sigma \qquad\quad \varepsilon^2 \qquad\quad \sigma \qquad\quad \varepsilon  \qquad\quad \varepsilon^2 \qquad\quad \varepsilon^3
\]
Therefore,
by an iterative application of rule (1) from Definition
\ref{tutte.pol}, we have
\[
T(M; x, y) = \hspace{0.1cm} y^2 \hspace{0.1cm} + \hspace{0.1cm} xy \hspace{0.1cm} + \hspace{0.1cm} xy \hspace{0.1cm} + \hspace{0.1cm} x^2 \hspace{0.1cm} + 
\hspace{0.1cm} y \hspace{0.1cm} + \hspace{0.1cm}  x \hspace{0.1cm} + \hspace{0.1cm} x^2 \hspace{0.1cm} + \hspace{0.1cm} x^3.
\]
\end{tutte.pol.ex}

\section{Covering family structures on the category of matroids} \label{cov.fam.sec}

\subsection{Outline of the main proof} \label{outline} 
We will start this section by introducing one of the two covering family structures on 
$\mathbf{Mat}_{+}$ that we discussed in the introduction. 

\theoremstyle{definition} \newtheorem{iso.cov.str}{Definition}[section]

\begin{iso.cov.str} \label{iso.cov.str}
Let $\mathcal{S}^{\cong}$ be the covering family structure on 
$\mathbf{Mat}_{+}$ defined as follows: 

\begin{itemize}

\vspace{0.2cm}

\item[(i)] Every finite (possibly empty) family $\{\ast \rightarrow \ast\}_{i\in I}$
 is in $\mathcal{S}^{\cong}$. Recall that $\ast$
 is the base-point object in $\mathbf{Mat}_{+}$.  
 
 \vspace{0.2cm}

\item[(ii)] If $f: M \stackrel{\cong}{\longrightarrow} N$ 
is an arbitrary isomorphism in $\mathbf{Mat}_{+}$,
then the singleton 
\[
\{ f: M \stackrel{\cong}{\longrightarrow} N \}
\]
belongs to $\mathcal{S}^{\cong}$.   

\end{itemize}

\vspace{0.2cm}

As mentioned in the introduction, we shall denote the category with 
covering families $(\mathbf{Mat}_{+}, \mathcal{S}^{\cong}, \ast)$
by $\mathbf{Mat}^{\cong}$. 

\end{iso.cov.str}

\vspace{0.2cm} 

As also discussed in the introduction, Theorem \ref{k0.result}
immediately implies that 
\[
K_0(\mathbf{Mat}^{\cong}) = \mathbb{Z}[\mathcal{M}],
\]
where $\mathcal{M}$ represents the set of isomorphism classes
of matroids. Most of the rest of this section shall be devoted
to the proof of Theorem \ref{thma}. 
We will break down this proof as follows: 

\begin{itemize}

\vspace{0.2cm}

\item[\textit{Step 1:}] Construct the covering family structure $\mathcal{S}^{\mathrm{tc}}$ on $\mathbf{Mat}_{+}$. 
As mentioned in the introduction, the covering families in $\mathcal{S}^{\mathrm{tc}}$ shall be referred to as \textit{Tutte coverings}. 

\vspace{0.2cm}

\item[\textit{Step 2:}] Verify that there is a canonical isomorphism 
$\rho: K_0(\mathbf{Mat}^{\mathrm{tc}}) \stackrel{\cong}{\longrightarrow} \mathbb{Z}[x,y]$,
where $\mathbf{Mat}^{\mathrm{tc}}$ denotes the triple 
$(\mathbf{Mat}_{+},\mathcal{S}^{\mathrm{tc}}, \ast)$. 

\vspace{0.2cm}

\item[\textit{Step 3:}] Show that the composition of $\gamma: K_0(\mathbf{Mat}^{\cong}) \longrightarrow K_0(\mathbf{Mat}^{\mathrm{tc}})$
and $\rho: K_0(\mathbf{Mat}^{\mathrm{tc}}) \stackrel{\cong}{\longrightarrow} \mathbb{Z}[x,y]$
maps any generator $[M] \in K_0(\mathbf{Mat}^{\cong})$ to its Tutte polynomial $T(M;x,y)$.
\end{itemize}

\vspace{0.2cm}

\subsection{The family of Tutte coverings}

In this subsection, we will construct the covering family structure 
$\mathcal{S}^{\mathrm{tc}}$ indicated in Step 1 above. Constructing
this structure, as well as establishing its key properties, 
is the main step for proving Theorem \ref{thma}.
Our overarching strategy for constructing the 
covering families in $\mathcal{S}^{\mathrm{tc}}$ is to
produce diagrams in $\mathbf{Mat}_{+}$ involving
deletions and contractions whose shape resembles that
of the tree-shaped diagram displayed in (\ref{first.tree}) in
Example \ref{tutte.pol.ex}.
In Definition \ref{treecat} below, we will introduce the class of diagrams 
in $\mathbf{Mat}_{+}$ we will use for this purpose.   
Before doing so, we need to discuss a few preliminaries.  

\theoremstyle{definition}     \newtheorem{tree}[iso.cov.str]{Conventions}

\begin{tree} \label{tree}
Consider a rooted binary tree $T$, as shown in the left-hand figure below. 

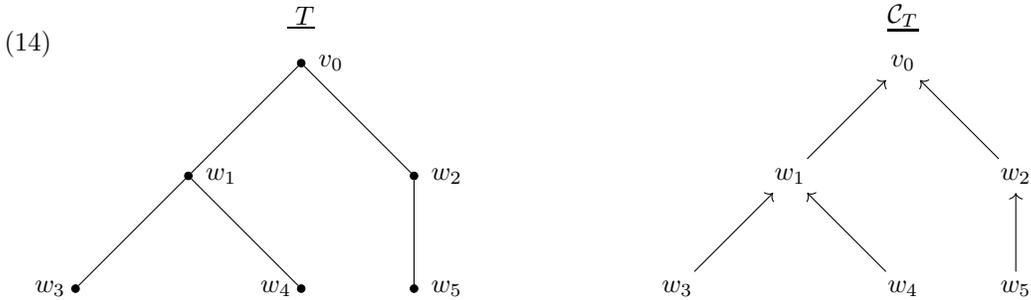
\begin{equation} \label{bin.tree.cat}
\quad
\end{equation}
\vspace{-1.4cm}
\[
\begin{tikzpicture}
      \tikzset{enclosed/.style={draw, circle, inner sep=0pt, minimum size=.1cm, fill=black}}
      
      \node (T) at (-4,0.6) {\underline{\hspace{0.1cm}$T$}};
      \node[enclosed, label={right, xshift=.05cm: $v_0$}] (v) at (-4, 0) {};
      \node[enclosed, label={right, xshift=.05cm: $w_1$}] (w1) at (-5.5, -1.5) {};
       \node[enclosed, label={right, xshift=.05cm: $w_2$}] (w2) at (-2.5, -1.5) {};
       \node[enclosed, label={right, xshift=.05cm: $w_5$}] (w5) at (-2.5,-3) {};
       \node[enclosed, label={left, xshift=.05cm: $w_3$}] (w3) at (-7,-3) {};
        \node[enclosed, label={left, xshift=.05cm: $w_4$}] (w4) at (-4,-3) {};
      
      \node (CT) at (4,0.6) {\underline{$\mathcal{C}_T$}};    
      \node (v1) at (4, 0) {$v_0$};
      \node (w11) at (2.5, -1.5) {$w_1$};
       \node (w22) at (5.5, -1.5) {$w_2$};
       \node  (w55) at (5.5,-3) {$w_5$};
       \node (w33) at (1,-3) {$w_3$};
       \node (w44) at (4,-3) {$w_4$};

          \draw (v) -- (w1);
          \draw (v) -- (w2);
          \draw (w1) -- (w3);
           \draw (w1) -- (w4);
           \draw (w2) -- (w5);
           
           \draw (w11) edge[->] (v1);
           \draw (w22) edge[->] (v1);
            \draw (w55) edge[->] (w22);
             \draw (w33) edge[->] (w11);
              \draw (w44) edge[->] (w11);
      \end{tikzpicture}
\]
As usually done in the literature, we shall call 
 the top-most vertex of a rooted binary tree the \textit{root} of the tree. 
 Also, given a vertex $v$ of a rooted binary tree, any 
 vertex sitting below $v$ and which can be connected to $v$ 
 via a sequence of edges shall be called a \textit{descendant} of $v$. For example,
 in the tree $T$ shown in (\ref{bin.tree.cat}), $w_3$ is a descendant of the vertex $v_0$. 
 Furthermore, the immediate descendants of a vertex 
 (i.e., those which are connected to it via a single edge) 
 are called the \textit{children} of the vertex. For example, 
 referring again to the tree $T$ in (\ref{bin.tree.cat}), 
 $w_1$ and $w_2$ are the children of the vertex $v_0$.
 On the other hand, $w_5$ is the only child of $w_2$.  
 Finally, vertices
 without any descendants are called \textit{leaves}, and any 
 vertex which is not a leaf is an
 \textit{internal vertex.} 
 Throughout the rest of this article, we shall only consider rooted binary trees 
 with finitely many vertices. Also, we shall typically denote the 
 root of a rooted binary tree $T$ by $\bullet_T$.    

\end{tree}

\vspace{0.2cm} 

The following definition describes the shape of the diagrams we shall consider
when defining Tutte coverings.  

\theoremstyle{definition}     \newtheorem{treecat}[iso.cov.str]{Definition} 

\begin{treecat} \label{treecat}
Fix a rooted binary tree $T$ and let $\mathrm{Vert}(T)$ be its
set of vertices. \textit{The category induced by $T$}, denoted
by $\mathcal{C}_T$, is the category determined by the partial order $\leq_T$
on $\mathrm{Vert}(T)$ defined by $w \leq_T v$ if and only if $w$ is a 
descendant of $v$ or $w = v$.
Referring again to the tree $T$ in (\ref{bin.tree.cat}),
the right-hand figure in (\ref{bin.tree.cat}) illustrates 
the category $\mathcal{C}_T$ induced by $T$. 
\end{treecat}

\vspace{0.2cm}

The following kinds of $T$-shaped diagrams in $\mathbf{Mat}_{+}$,
where $T$ is an arbitrary rooted binary tree,
shall be the basic building blocks we will use to construct
the desired covering family structure $\mathcal{S}^{\mathrm{tc}}$ on $\mathbf{Mat}_{+}$.  

\theoremstyle{definition}     \newtheorem{treediag}[iso.cov.str]{Definition}

\begin{treediag} \label{treediag}
Consider a rooted binary tree $T$. 
An \textit{elementary deletion-contraction tree of shape $T$}
is a functor of the form 
$\mathcal{F}: \mathcal{C}_T \rightarrow \mathbf{Mat}_{+}$ 
with the following properties:  

\begin{itemize}

\vspace{0.2cm}

\item[(i)] $\mathcal{F}(\bullet_T) \neq \ast$, where $\ast$ is the base-point object of $\mathbf{Mat}_{+}$. 
 
\vspace{0.2cm}

\item[(ii)] If $v$ is an internal vertex  of $T$ with only one child $w$, then 
we must have that $\mathcal{F}(w) = \mathcal{F}(v)$ and 
$\mathcal{F}(w \rightarrow v) = \mathrm{Id}_{\mathcal{F}(v)}$.

\vspace{0.2cm}

\item[(iii)] If $v$ is an internal vertex of $T$ with two children $w_1$ and $w_2$,  
then there is an element $e$ 
of the matroid $N := \mathcal{F}(v)$ for which the following holds: 

\begin{itemize}

\vspace{0.15cm}

\item[$\cdot$] $e$ is a non-degenerate element of $N$ (see Definition \ref{non.deg}).

\vspace{0.15cm}

\item[$\cdot$] $\mathcal{F}$ maps the subdiagram 
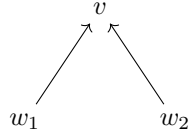
\[
\begin{tikzpicture}
      \tikzset{enclosed/.style={draw, circle, inner sep=0pt, minimum size=.1cm, fill=black}}
      
      \node (v) at (0,0) {$v$};
       \node (w1) at (-1,-1.5) {$w_1$};
      \node (w2) at (1,-1.5) {$w_2$};
      
      \draw (w1) edge[->] (v);
      \draw (w2) edge[->] (v);
      \end{tikzpicture}
\]  
of $\mathcal{C}_T$ to one of the following two diagrams: 
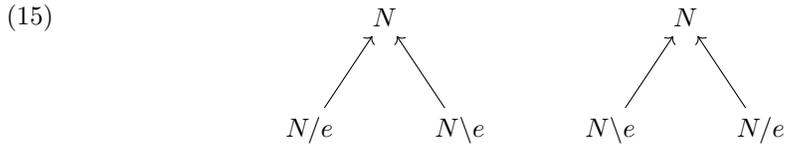
\begin{equation} \label{splitting}
\quad
\end{equation}
\vspace{-1cm}
\[
\begin{tikzpicture}
      \tikzset{enclosed/.style={draw, circle, inner sep=0pt, minimum size=.1cm, fill=black}}
      
      \node (v0) at (-2,0) {$N$};
       \node (w1) at (-3,-1.5) {$N/e$};
      \node (w2) at (-1,-1.5) {$N\backslash e$};
      
      \draw (w1) edge[->] (v0);
      \draw (w2) edge[->] (v0);
      
      \node (v1) at (2,0) {$N$};
       \node (w11) at (1,-1.5) {$N\backslash e$};
      \node (w22) at (3,-1.5) {$N/e$};
      
        \draw (w11) edge[->] (v1);
      \draw (w22) edge[->] (v1);
      \end{tikzpicture}
\]  
In both of these diagrams, the diagonal maps are the 
standard inclusions from $N/e$ and $N\backslash e$ into $N$
(see Definition \ref{catmat}). 
A diagram of the form displayed in (\ref{splitting}) shall be called
a \textit{splitting of} $\mathcal{F}$.   

\end{itemize} 

\vspace{0.2cm} 
 
\end{itemize}

If $w \in T$ is a descendant of a vertex $v\in T$, then
it follows from the previous three conditions that
the matroid $\mathcal{F}(w)$ is a minor of $\mathcal{F}(v)$ and that $\mathcal{F}$
maps the unique morphism $w \rightarrow v$ in $\mathcal{C}_T$ to the standard inclusion 
$\mathcal{F}(w) \hookrightarrow \mathcal{F}(v)$.
\end{treediag}

We can generalize the previous definition as follows.

\theoremstyle{definition} \newtheorem{treediag.cont}[iso.cov.str]{Definition}

\begin{treediag.cont} \label{treediag.cont}

We say that $\mathcal{F}: \mathcal{C}_T \rightarrow \mathbf{Mat}_{+}$ is a 
\textit{deletion-contraction tree of shape $T$} if it is naturally isomorphic to an elementary 
deletion-contraction tree $\mathcal{G}$ of shape $T$, i.e., there is a natural transformation
$\eta: \mathcal{F} \Rightarrow \mathcal{G}$ whose components are all isomorphisms. 
Moreover, if $\mathcal{F}: \mathcal{C}_T \rightarrow \mathbf{Mat}_{+}$ is a deletion-contraction tree of shape $T$ with
$\mathcal{F}(\bullet_T) = M$, we shall sometimes say that $\mathcal{F}$
is \textit{rooted at }$M$.  
\end{treediag.cont}

\theoremstyle{definition} \newtheorem{examp.tree.diag2}[iso.cov.str]{Example} 

\begin{examp.tree.diag2} \label{examp.tree.diag2}
If $T$ is a single edge (i.e., a root $\bullet_T$ with a single child $w$), then a deletion-contraction tree
of shape $T$ is just an isomorphism $N \stackrel{\cong}{\longrightarrow} M$, 
where $M = \mathcal{F}(\bullet_T)$ and $N = \mathcal{F}(w)$.  
\end{examp.tree.diag2}

\theoremstyle{definition} \newtheorem{examp.tree.diag}[iso.cov.str]{Example} 

\begin{examp.tree.diag} \label{examp.tree.diag}
Consider the following graphs $G$, $G'$, and $G''$: 
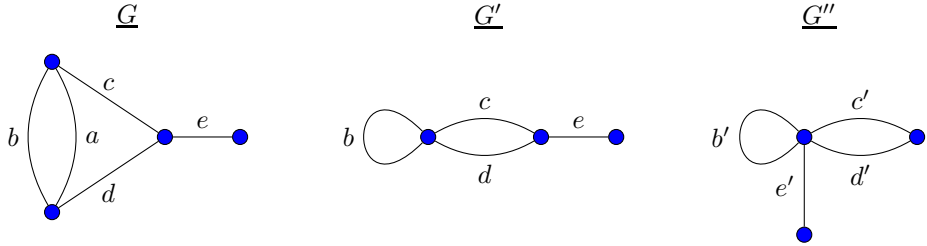
\vspace{0.3cm}
 \[
\begin{tikzpicture}
      \tikzset{enclosed/.style={draw, circle, inner sep=0pt, minimum size=.2cm, fill=blue}}
      
      \node (v0) at (-5, 0.6) {\underline{$G$}};
      \node[enclosed] (v1) at (-6,0) {};
      \node[enclosed] (v2) at (-6,-2) {};
      \node[enclosed] (v3) at (-4.5, -1) {};
      \node[enclosed] (v4) at (-3.5, -1) {};
      
      \draw [-] (v2) to [out = 120, in=240] node[left,midway] {$b$} (v1);
      \draw [-] (v2) to [out = 60, in=300] node[right,midway] {$a$} (v1);
      \draw (v1) -- (v3) node[midway, above] {$c$};
      \draw (v2) -- (v3) node[midway, below] {$d$};
      \draw (v3) -- (v4) node[midway, above] {$e$};
      
       \node (v00) at (-0.2, 0.6) {\underline{$G'$}};
      \node[enclosed] (v11) at (-1,-1) {};
      \node[enclosed] (v33) at (0.5,-1) {};
      \node[enclosed] (v44) at (1.5, -1) {};
      
      \draw (v11) to [out = 135, in= 225, looseness = 25] node[left,midway] {$b$} (v11);
      \draw (v11) to [out = 30, in= 150] node[above,midway] {$c$} (v33);
      \draw (v11) to [out = -30, in= -150] node[below,midway] {$d$} (v33);
      \draw (v33) -- (v44) node[midway, above] {$e$};
       
       \node (w00) at (4.2, 0.6) {\underline{$G''$}}; 
      \node[enclosed] (w11) at (4,-1) {};
      \node[enclosed] (w33) at (5.5,-1) {};
      \node[enclosed] (w44) at (4, -2.3) {};
      
      \draw (w11) to [out = 135, in= 225, looseness = 25] node[left,midway] {$b'$} (w11);
      \draw (w11) to [out = 30, in= 150] node[above,midway] {$c'$} (w33);
      \draw (w11) to [out = -30, in= -150] node[below,midway] {$d'$} (w33);
      \draw (w11) -- (w44) node[midway, left] {$e'$};
      
\end{tikzpicture}
\]
We shall denote the graphical matroids induced by these three graphs 
by $M$, $M'$, and $N$ respectively. 
Note that $G'$ is obtained by contracting the edge $a$ in $G$.
Therefore, $M' = M/a$. Also, there is an evident isomorphism 
$M' \cong N$ between the matroids induced by $G'$ and $G''$.  
Now, consider the following rooted binary trees $T_1$ and $T_2$: 
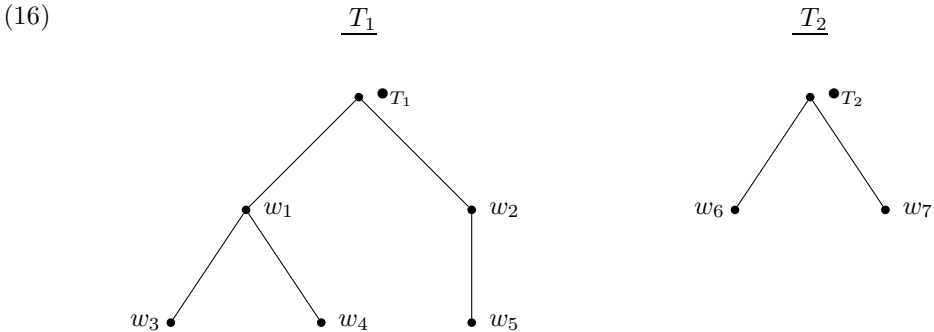
\begin{equation} \label{examp.trees}
\quad
\end{equation} 
\vspace{-1cm}
\[
\begin{tikzpicture}
      \tikzset{enclosed/.style={draw, circle, inner sep=0pt, minimum size=.1cm, fill=black}}
      
      \node (x) at (-3,1) {\underline{\hspace{0.1cm}$T_1$}}; 
      \node[enclosed, label={right, xshift=.05cm: $\bullet_{T_1}$}] (v0) at (-3, 0) {};
      \node[enclosed, label={right, xshift=.05cm: $w_1$}] (w1) at (-4.5, -1.5) {};
      \node[enclosed, label={right, xshift=.05cm: $w_2$}] (w2) at (-1.5, -1.5) {};
      \node[enclosed, label={left, xshift=.05cm: $w_3$}] (w3) at (-5.5,-3) {};
      \node[enclosed, label={right, xshift=.05cm: $w_4$}] (w4) at (-3.5,-3) {};
      \node[enclosed, label={right, xshift=.05cm: $w_5$}] (w5) at (-1.5, -3) {};
        
      \draw (v0) -- (w1);
      \draw (v0) -- (w2);
      \draw (w1) -- (w3);
      \draw (w1) -- (w4);  
      \draw (w2) -- (w5); 
      
      \node (y) at (3,1) {\underline{\hspace{0.1cm}$T_2$}};
      \node[enclosed, label={right, xshift=.05cm: $\bullet_{T_2}$}] (v00) at (3, 0) {};
      \node[enclosed, label={left, xshift=.05cm: $w_6$}] (w6) at (2, -1.5) {};
      \node[enclosed, label={right, xshift=.05cm: $w_7$}] (w7) at (4, -1.5) {};

      \draw (v00) -- (w6);
      \draw (v00) -- (w7);

\end{tikzpicture}
\]
Figure (\ref{examp.funct}) below displays examples of deletion-contraction trees $\mathcal{F}_1: \mathcal{C}_{T_1} \rightarrow \mathbf{Mat}_{+}$
and $\mathcal{F}_2: \mathcal{C}_{T_2} \rightarrow \mathbf{Mat}_{+}$ rooted at the matroids $M$ and $N$ respectively:   
\newpage
\begin{equation} \label{examp.funct}
\quad   
\end{equation} 
\vspace{-1cm}
\[
\begin{tikzpicture}
      \tikzset{enclosed/.style={draw, circle, inner sep=0pt, minimum size=.1cm, fill=black}}
      
      \node (x) at (-3,1) {\underline{\hspace{0.1cm}$\mathcal{F}_1$}}; 
      \node (v0) at (-3, 0) {$M$};
      \node (w1) at (-4.5, -1.5) {$M \backslash a$};
      \node (w2) at (-1.5, -1.5) {$M/a$};
      \node (w3) at (-5.5,-3) {$M\backslash a/b$};
      \node (w4) at (-3.5,-3) {$M\backslash a \backslash b$};
      \node (w5) at (-1.5, -3) {$N$}; 
      
       \draw (w1) edge[->] (v0);
       \draw (w2) edge[->] (v0);
       \draw (w3) edge[->] (w1);
       \draw (w4) edge[->] (w1);
       \draw (w5) edge[->] node[midway, right] {$\cong$} (w2);
       
      \node (y) at (3,1) {\underline{\hspace{0.1cm}$\mathcal{F}_2$}};
      \node (v00) at (3, 0) {$N$};
      \node (w6) at (2, -1.5) {$N/ c'$};
      \node (w7) at (4, -1.5) {$N \backslash c'$};
      
      \draw (w6) edge[->] (v00);
      \draw (w7) edge[->] (v00);
       
\end{tikzpicture}
\]
Note that every splitting in $\mathcal{F}_1$ and $\mathcal{F}_2$ is obtained by deleting and contracting a
non-degenerate element of the matroid located at the corresponding node.    
The vertical morphism in the left-hand diagram represents an isomorphism between
$N$ and $M/a = M'$. All other morphisms in (\ref{examp.funct}) are standard inclusions
(in the sense of Definition \ref{catmat}).  
\end{examp.tree.diag}

The reader may have noticed that it is possible to merge the diagrams displayed in 
(\ref{examp.funct}) to produce a larger deletion-contraction tree. Namely, since 
the matroid $N$
is both a leaf in $\mathcal{F}_1$ and the root
of $\mathcal{F}_2$, it is possible to merge the two diagrams at $N$ to produce
the following deletion-contraction tree: 
\begin{equation} \label{examp.funct.merge}
\quad 
\end{equation} 
\vspace{-1cm}
\[
\begin{tikzpicture}
      \tikzset{enclosed/.style={draw, circle, inner sep=0pt, minimum size=.1cm, fill=black}}
      
      \node (x) at (-3,1) {\underline{\hspace{0.1cm}$\mathcal{F}$}}; 
      \node (v0) at (-3, 0) {$M$};
      \node (w1) at (-4.5, -1.5) {$M \backslash a$};
      \node (w2) at (-1.5, -1.5) {$M/a$};
      \node (w3) at (-5.5,-3) {$M\backslash a/b$};
      \node (w4) at (-3.5,-3) {$M\backslash a \backslash b$};
      \node (w5) at (-1.5, -3) {$N$}; 
      
       \draw (w1) edge[->] (v0);
       \draw (w2) edge[->] (v0);
       \draw (w3) edge[->] (w1);
       \draw (w4) edge[->] (w1);
       \draw (w5) edge[->] node[midway, right] {$\cong$} (w2);
       
      \node (w6) at (-2.5, -4.5) {$N/ c'$};
      \node (w7) at (-0.5, -4.5) {$N \backslash c'$};
      
      \draw (w6) edge[->] (w5);
      \draw (w7) edge[->] (w5);
\end{tikzpicture}
\]       
The above diagram represents a functor of the form $\mathcal{F}: \mathcal{C}_T \rightarrow \mathbf{Mat}_{+}$, where
$\mathcal{C}_T$ is the category induced by the rooted binary tree $T$ obtained
by gluing the trees $T_1$ and $T_2$ from (\ref{examp.trees}) at the vertices $w_5$ and  $\bullet_{T_2}$. 
We formalize this construction in the next definition. 

\theoremstyle{definition} \newtheorem{treediag.merge}[iso.cov.str]{Definition}

\begin{treediag.merge} \label{treediag.merge}
Consider two deletion-contraction trees 
$\mathcal{F}_1: \mathcal{C}_{T_1} \rightarrow \mathbf{Mat}_{+}$
and $\mathcal{F}_2: \mathcal{C}_{T_2} \rightarrow \mathbf{Mat}_{+}$ 
such that $\mathcal{F}_1(v) = \mathcal{F}_2(\bullet_{T_2})$
for some leaf $v$ of $T_1$.  If $T$ is the rooted binary
tree obtained by gluing $T_1$ and $T_2$ at the points 
$v$ and $\bullet_{T_2}$, then we can define a functor 
$\mathcal{F}: \mathcal{C}_T \rightarrow \mathbf{Mat}_{+}$ 
as follows:  

\begin{itemize}

\vspace{0.2cm}

\item[(i)] Without loss of generality, identify
the trees $T_1$ and $T_2$ with the subtrees of 
$T$ obtained by taking the images of the obvious
inclusions $T_1 \hookrightarrow T$ and 
$T_2 \hookrightarrow T$ respectively. Similarly, 
identify the categories 
$\mathcal{C}_{T_1}$ and $\mathcal{C}_{T_2}$ 
with the obvious subcategories of $\mathcal{C}_T$. 
Then, with these identifications, we define the functor 
$\mathcal{F}$ on $\mathcal{C}_{T_1}$ and $\mathcal{C}_{T_2}$ as  

\[
\mathcal{F}|_{\mathcal{C}_{T_1}} := \mathcal{F}_1 \quad\text{ and }\quad
\mathcal{F}|_{\mathcal{C}_{T_2}} := \mathcal{F}_2. 
 \]
 
\item[(ii)] Next, consider two vertices $w_1 \in T_1$ and $w_2 \in T_2$. 
If $w_2$ is a descendant of $w_1$ in $T$, we define 
$\mathcal{F}(w_2 \rightarrow w_1)$ to be the morphism obtained by taking the composition 

\[
\mathcal{F}_1(v \rightarrow w_1) \circ \mathcal{F}_2(w_2 \rightarrow \bullet_{T_2}).
\] 
  
\end{itemize}

\vspace{0.2cm}

It is straightforward to verify that  
$\mathcal{F}: \mathcal{C}_T \rightarrow \mathbf{Mat}_{+}$
is a deletion-contraction tree.
In this case, we say that $\mathcal{F}$ was
obtained by \textit{attaching $\mathcal{F}_2$ to $\mathcal{F}_1$ at the leaf $v$}.  
\end{treediag.merge}

\theoremstyle{definition} \newtheorem{treediag.merge.rem}[iso.cov.str]{Remark}

\begin{treediag.merge.rem} \label{treediag.merge.rem}

Evidently, it is possible to extend the previous definition to the case 
when we have more than two trees. More precisely, let
$\mathcal{F}_0$, $\mathcal{F}_1$, \ldots, $\mathcal{F}_p$
be deletion-contraction trees of shape $T_0$, $T_1$, \ldots, $T_p$
respectively. Moreover, suppose that there are $p$ distinct leaves
$v_1$, \ldots, $v_p$ in $T_0$ for which we have 
$\mathcal{F}_0(v_j) = \mathcal{F}_j(\bullet_{T_j})$ for 
$j = 1, \ldots, p$. Then, by repeating the construction 
introduced in the previous definition $p$ times, we can 
produce a new deletion-contraction tree 
$\mathcal{F}: \mathcal{C}_T \rightarrow \mathbf{Mat}_{+}$
by attaching  $\mathcal{F}_1$, \ldots, $\mathcal{F}_p$
to $\mathcal{F}_0$ at the leaves 
$v_1$, \ldots, $v_p$ respectively. 
\end{treediag.merge.rem}

\vspace{0.2cm}

We need one more ingredient before we can define the 
covering family structure $\mathcal{S}^{\mathrm{tc}}$ on $\mathbf{Mat}_{+}$.  
For this next definition, we shall adopt the following notation: If $v$ is a vertex in a rooted binary tree $T$ with 
root $\bullet_T$, then we denote the unique morphism $v \rightarrow \bullet_{T}$ in $\mathcal{C}_T$ by $i_{v}$.

\vspace{0.2cm}

\theoremstyle{definition} \newtheorem{leaf.root}[iso.cov.str]{Definition}

\begin{leaf.root} \label{leaf.root}
Consider a rooted binary tree $T$ and let $v_1, \ldots, v_p$
be the leaves of $T$. Given a deletion-contraction tree
$\mathcal{F}: \mathcal{C}_T \rightarrow \mathbf{Mat}_{+}$, the multi-morphism 
\[
\big\{ \mathcal{F}(i_{v_j}): \mathcal{F}(v_j) \rightarrow \mathcal{F}(\bullet_T) \big\}_{j = 1,\ldots, p} 
\]
shall be called \textit{the collection of leaf-to-root morphisms of} $\mathcal{F}$.     
\end{leaf.root}

\vspace{0.2cm}

We are now ready to define our second covering family structure on $\mathbf{Mat}_{+}$. 

\theoremstyle{definition} \newtheorem{tc.cov.str}[iso.cov.str]{Definition}

\begin{tc.cov.str} \label{tc.cov.str}
Let $\mathcal{S}^{\mathrm{tc}}$ be the collection of
multi-morphisms in $\mathbf{Mat}_{+}$ defined
as follows: 

\begin{itemize}

\vspace{0.15cm}

\item[(i)] Every finite (possibly empty) family $\{\ast \rightarrow \ast\}_{i\in I}$
 is in $\mathcal{S}^{\mathrm{tc}}$. Once again, $\ast$ represents
 the base-point object in $\mathbf{Mat}_{+}$.

 \vspace{0.15cm}

\item[(ii)] Given a matroid $M \neq \ast$, a multi-morphism of the form
$\{ f_j: N_j \rightarrow M\}_{j = 1,\ldots, p} $  
belongs to $\mathcal{S}^{\mathrm{tc}}$ if and only if
it can be realized as the collection of leaf-to-root morphisms of
some deletion-contraction tree.
In other words, $\{ f_j: N_j \rightarrow M\}_{j = 1,\ldots, p} $ 
is in $\mathcal{S}^{\mathrm{tc}}$ if and only if 
there exists a rooted binary tree $T$ with leaves $v_1, \ldots, v_p$ 
and a deletion-contraction tree $\mathcal{F}: \mathcal{C}_T \rightarrow \mathbf{Mat}_{+}$ 
rooted at $M$ such that $f_j = \mathcal{F}(i_{v_j})$ for each $j = 1, \ldots, p$.
Recall that $i_{v_j}$ is the unique morphism 
$v_j \rightarrow \bullet_T$ in $\mathcal{C}_T$. 
\end{itemize}

 \vspace{0.15cm}

As mentioned earlier, a multi-morphism
belonging to $\mathcal{S}^{\mathrm{tc}}$ 
of the form $\{ f_j: N_j \rightarrow M\}_{j = 1,\ldots, p}$
shall be called a \textit{Tutte covering of $M$}.  
Moreover, as we have indicated already,  
the category with covering families 
$(\mathbf{Mat}_{+}, \hspace{0.05cm} \mathcal{S}^{\mathrm{tc}}, \ast)$
will be denoted by $\mathbf{Mat}^{\mathrm{tc}}$.  
\end{tc.cov.str}

\vspace{0.2cm}

Our next step is to verify
that $\mathcal{S}^{\mathrm{tc}}$ is indeed a covering
family structure.

\theoremstyle{plain} \newtheorem{cov.mat.prop}[iso.cov.str]{Proposition} 

\begin{cov.mat.prop} \label{cov.mat.prop}
The collection $\mathcal{S}^{\mathrm{tc}}$ is a covering family structure 
on the category $\mathbf{Mat}_{+}$. 
\end{cov.mat.prop}

\begin{proof}
By part (i) of the
previous definition, we have immediately that 
$\mathcal{S}^{\mathrm{tc}}$ satisfies condition (a)
from Definition \ref{catcov}.  
On the other hand, 
as remarked in Example \ref{examp.tree.diag2},
any isomorphism $N \stackrel{\cong}{\longrightarrow} M$
is a deletion-contraction tree. In particular,
any identity map $\mathrm{Id}_M: M \rightarrow M$ defines
a deletion-contraction tree,
and it follows that 
every singleton of the form $\{\mathrm{Id}_M: M \rightarrow M\}$ is 
in $\mathcal{S}^{\mathrm{tc}}$. To show that 
$\mathcal{S}^{\mathrm{tc}}$ also satisfies condition 
(c) from Definition \ref{catcov}, take a collection 
of $p+1$ multi-morphisms 
in $\mathcal{S}^{\mathrm{tc}}$ of the following form: 
\[
\{g_j: M_j \rightarrow M\}_{j \in \{1,\ldots,p\}} \qquad \{f_{i1}: N_{i1} \rightarrow M_1\}_{i \in I_1} \quad \ldots\ldots \quad \{f_{ip}: N_{ip} \rightarrow M_p\}_{i \in I_p}.
\]
By part (ii) of Definition \ref{tc.cov.str}, we can realize each of these multi-morphisms as
the collection of leaf-to-root morphisms of deletion-contraction trees
\[
\mathcal{F}_0 \quad \mathcal{F}_1 \quad \ldots\ldots \quad \mathcal{F}_p
\]
of shape $T_0$, $T_1$, \ldots, $T_p$ respectively. In particular, if 
$v_1$, \ldots, $v_p$ are the leaves of $T_0$, 
we have 
\[
\mathcal{F}_j(\bullet_{T_j}) = M_j = \mathcal{F}_0(v_j)
\]
for each $j = 1, \ldots, p$.  
It is then straightforward to verify that 
the collection of compositions
\begin{equation} \label{comp.output}
\{ g_j \circ f_{ij}: N_{ij} \longrightarrow M \}_{j\in \{1,\ldots,p\}, \hspace{0.1cm} i \in I_j}
\end{equation}
is equal to the collection of leaf-to-root morphisms of the deletion-contraction
tree $\mathcal{F}$ obtained by attaching 
$\mathcal{F}_1$, \ldots. $\mathcal{F}_p$ to $\mathcal{F}_0$ at the 
leaves $v_1$, \ldots, $v_p$
(see Remark \ref{treediag.merge.rem}). 
 Therefore, the collection given in 
(\ref{comp.output}) is also in $\mathcal{S}^{\mathrm{tc}}$,
and we have thus shown that 
 $\mathcal{S}^{\mathrm{tc}}$ satisfies 
 all the requirements for being a 
 covering family structure. 
\end{proof}

\theoremstyle{definition} \newtheorem{cov.mat.morhp}[iso.cov.str]{Remark} 

\begin{cov.mat.morhp} \label{cov.mat.morhp}
Note that any covering family in $\mathbf{Mat}^{\cong}$
is also a covering family in  $\mathbf{Mat}^{\mathrm{tc}}$. It follows that
the identity functor $\mathbf{Mat}_{+} \stackrel{=}{\longrightarrow} \mathbf{Mat}_{+}$
underlies a morphism 
\begin{equation} \label{Gamma.2}
\Gamma: \mathbf{Mat}^{\cong} \longrightarrow \mathbf{Mat}^{\mathrm{tc}}
\end{equation}
of categories with covering families. As indicated in the introduction, 
we shall denote the map 
$K_0(\mathbf{Mat}^{\cong}) \rightarrow K_0(\mathbf{Mat}^{\mathrm{tc}})$
between $K_0$ groups corresponding to the morphism 
$\Gamma$ in (\ref{Gamma.2}) by $\gamma$. 
We point out that this homomorphism $\gamma$ is a quotient map. 
Indeed, as remarked earlier, 
$K_0(\mathbf{Mat}^{\cong})$ is the free abelian group 
$\mathbb{Z}[\mathcal{M}]$ on the set $\mathcal{M}$
of isomorphism classes of matroids. 
On the other hand, by Theorem \ref{k0.result}, the group
$K_0(\mathbf{Mat}^{\mathrm{tc}})$
is the quotient of $\mathbb{Z}[\mathcal{M}]$ modulo the subgroup 
$H < \mathbb{Z}[\mathcal{M}]$
generated by elements of the form 
\[
[M] -\big( [N_1] + \ldots\ldots + [N_p] \big),
\]
where $N_1$, \ldots, $N_p$ are the domains 
of a Tutte covering for $M$. 
It is now evident that 
the homomorphism $\gamma$
agrees with the quotient map 
$\mathbb{Z}[\mathcal{M}] \rightarrow \mathbb{Z}[\mathcal{M}]/H$.  
\end{cov.mat.morhp}

Proposition \ref{cov.mat.prop} completes Step 1 of the proof of Theorem \ref{thma}
(see the outline given at the end of \S \ref{outline}). Now we will address the second
step, i.e., we will describe the isomorphism type 
of the group $K_0(\mathbf{Mat}^{\mathrm{tc}})$.
By the way we defined the covering family structure $\mathcal{S}^{\mathrm{tc}}$, 
it follows that the only indecomposable objects in $\mathbf{Mat}^{\mathrm{tc}}$
(in the sense of Definition \ref{ind}) are matroids which are 
isomorphic to finite direct sums of the form
\begin{equation} \label{ist.loop2}
\varepsilon^m \oplus \sigma^n.
\end{equation}
That is, $M$ is indecomposable in $\mathbf{Mat}^{\mathrm{tc}}$ if
it is the direct sum of finitely many isthmuses and loops. 
In the above direct sum, we may have $m=0$ or $n=0$. 
If $m$ and $n$ are both zero, (\ref{ist.loop2}) becomes the empty matroid $\varnothing$.   
It turns out that any object $M \neq \ast$ in $\mathbf{Mat}^{\mathrm{tc}}$
admits a Tutte covering consisting entirely of indecomposable objects.
We shall prove this fact in the next proposition. 

\theoremstyle{plain} \newtheorem{ind.refine}[iso.cov.str]{Proposition}

\begin{ind.refine} \label{ind.refine}
Any matroid $M \neq \ast$ admits a Tutte covering  
$\{ g_i: N_i \rightarrow M\}_{i \in I}$ where each $N_i$ is indecomposable
in $\mathbf{Mat}^{\mathrm{tc}}$.    
\end{ind.refine}

\begin{proof}
We shall prove this claim by induction on the number of
non-degenerate elements in $M$.\\

\textit{\underline{Base case:}} If $M$ has no non-degenerate elements 
(i.e., if each element of $M$ is either an isthmus or a loop),
then it suffices to take the covering $\{ \mathrm{Id}_M: M \rightarrow M\}$.  \\ 
 
\textit{\underline{Inductive step:}} Now, consider a matroid $M$ 
with exactly $p>0$ non-degenerate elements, and suppose that
the claim is true for all matroids with at most $p-1$ non-degenerate
elements. For the matroid $M$,  
fix a non-degenerate element $e$, and consider the following 
elementary deletion-contraction tree $\mathcal{F}$:
\begin{equation} \label{init.split}
\quad 
\end{equation}
\vspace{-1cm}
\[
\begin{tikzpicture}
      \tikzset{enclosed/.style={draw, circle, inner sep=0pt, minimum size=.1cm, fill=black}}
      
      \node (v) at (3, 0) {$M$};
      \node (w1) at (2, -1.5) {$M/ e$};
      \node (w2) at (4, -1.5) {$M \backslash e$};
      
      \draw (w1) edge[->] (v);
      \draw (w2) edge[->] (v);
       
\end{tikzpicture}
\]
Since the matroids $M/e$ and $M \backslash e$ can only have at most
$p-1$ non-degenerate elements, both of these matroids admit Tutte coverings 
\[
\{ f_{i}: N_{i} \rightarrow M/ e  \}_{i \in I}  \qquad  \{ f'_{j}: N'_{j} \rightarrow M\backslash e \}_{j\in J} 
\]
where each $N_{i}$ and $N'_{j}$ is indecomposable. 
Furthermore, both of these coverings
can be realized as the collections of leaf-to-root morphisms of deletion-contraction trees
$\mathcal{F}_1$ and $\mathcal{F}_2$ respectively. 
Then, by attaching
$\mathcal{F}_1$ and $\mathcal{F}_2$ to the leaves of the tree in (\ref{init.split}), 
we obtain a deletion-contraction tree $\mathcal{F}$ whose collection of 
leaf-to-root morphisms is the following: 
\[
\{ i_1\circ f_{i}: N_{i} \rightarrow M \}_{i \in I}  \hspace{0.1cm} \cup \hspace{0.12cm}  \{ i_2 \circ f'_{j}: N'_{j} \rightarrow M \}_{j \in J}. 
\]
Therefore, the matroid $M$ also admits a Tutte covering where each 
morphism has an indecomposable domain,
which is exactly what we wanted to prove. 
\end{proof}

\vspace{0.2cm}

The previous result motivates the following definition. 

\theoremstyle{definition} \newtheorem{complete.cov}[iso.cov.str]{Definition}

\begin{complete.cov} \label{complete.cov}

A Tutte covering $\Lambda = \{ f_j: N_j \rightarrow M\}_{j \in \{ 1, \ldots, p\}}$ for a matroid 
$M \neq \ast$ shall be called \textit{indecomposable}
if each domain $N_j$ is indecomposable. For such a covering, 
we will denote
the multi-set of domains $N_1, \ldots, N_p$ by $\mathrm{Ind}_M(\Lambda)$.    
\end{complete.cov}

\vspace{0.2cm}

Perhaps the most essential fact we need to prove in order to describe 
the isomorphism type of $K_0(\mathbf{Mat}^{\mathrm{tc}})$
is that, for any matroid $M \neq \ast$, 
the multi-set $\mathrm{Ind}_M(\Lambda)$ is (up to isomorphism) independent of 
the indecomposable Tutte covering $\Lambda$.  
We shall establish this fact in the next proposition.  

\theoremstyle{plain} \newtheorem{ind.ind}[iso.cov.str]{Proposition}

\begin{ind.ind} \label{ind.ind}
Fix a matroid $M \neq \ast$ in $\mathbf{Mat}_{+}$. 
If $\Lambda$ and $\Omega$ are 
indecomposable Tutte coverings of $M$, then 
$\mathrm{Ind}_M(\Lambda)$ 
and $\mathrm{Ind}_M(\Omega)$ are isomorphic
as multi-sets of matroids (see Definition \ref{iso.multi}).
\end{ind.ind}

\begin{proof}
We shall also prove this result by performing induction 
on the number of non-degenerate elements 
in $M$. To make our notation less cumbersome,
we shall drop the subscript in the notation $\mathrm{Ind}_M(\Lambda)$.
Thus, for an indecomposable Tutte covering $\Lambda$ of 
$M$, we shall denote the multi-set of domains in $\Lambda$
simply by $\mathrm{Ind}(\Lambda)$.  \\  

\textit{\underline{Base case:}} Suppose that $M$ does not have any non-degenerate
elements. In other words, $M$ consists only of isthmuses and loops, which means that $M$
must be isomorphic to a direct sum of the form
\[
\varepsilon^m \oplus \sigma^n. 
\]
Then, for any indecomposable covering $\Lambda$ 
of $M$, we must have that $\mathrm{Ind}(\Lambda)$ is
isomorphic to the multi-set $\{\varepsilon^m \oplus \sigma^n\}$
consisting only of the matroid $\varepsilon^m \oplus \sigma^n$.
It follows that the result holds for the base case.    \\   

\textit{\underline{Inductive step:}} Let $p \in \mathbb{Z}_{>0}$ and assume that the result is true
for all matroids with at most $p -1$ non-degenerate elements. Also, fix the following data: 

\begin{itemize}

\vspace{0.15cm}

\item[$\cdot$] A matroid $M$ with exactly $p$ non-degenerate elements. 

\vspace{0.15cm}

\item[$\cdot$] Two indecomposable Tutte coverings 
$\Lambda$ and $\Omega$ of $M$. 

\vspace{0.15cm}

\end{itemize}

Without loss of generality, we can assume
that $\Lambda$ and $\Omega$ are Tutte coverings 
induced by elementary deletion-contraction trees, i.e., 
there are elementary deletion-contraction 
trees $\mathcal{F}$ and $\mathcal{G}$ such that
$\Lambda$ and $\Omega$ are the 
collections of leaf-to-root morphisms of 
$\mathcal{F}$ and $\mathcal{G}$ respectively. 
With this assumption, 
the top two levels of $\mathcal{F}$ and $\mathcal{G}$
are respectively of the form 
\begin{equation} \label{top.levels}
\quad
\end{equation}
\vspace{-1cm}
\[
\begin{tikzpicture}
      \tikzset{enclosed/.style={draw, circle, inner sep=0pt, minimum size=.1cm, fill=black}}
      
      \node (x) at (-2,1cm) {$\underline{\mathcal{F}_0}$};
      \node (v0) at (-2,0) {$M$};
       \node (w1) at (-3,-1.5) {$M/e$};
      \node (w2) at (-1,-1.5) {$M\backslash e$};
      
      \draw (w1) edge[->] (v0);
      \draw (w2) edge[->] (v0);
      
       \node (y) at (2,1cm) {$\underline{\mathcal{G}_0}$};
      \node (v1) at (2,0) {$M$};
       \node (w11) at (1,-1.5) {$M / f$};
      \node (w22) at (3,-1.5) {$M\backslash f$};
      
        \draw (w11) edge[->] (v1);
      \draw (w22) edge[->] (v1);
      \end{tikzpicture}
\]    
for some 
non-degenerate elements $e$ and $f$ of $M$. 
If $e = f$, then our induction hypothesis implies that
$\mathrm{Ind}(\Lambda) \cong \mathrm{Ind}(\Omega)$. 
Thus, we shall assume from now on that $e \neq f$.  
As indicated in the figure above, we are denoting the trees in (\ref{top.levels}) 
by $\mathcal{F}_0$ and $\mathcal{G}_0$ respectively. 
The trees $\mathcal{F}$ and $\mathcal{G}$ are then
obtained by attaching  elementary deletion-contraction trees
$\mathcal{F}_1$, $\mathcal{F}_2$, $\mathcal{G}_1$, $\mathcal{G}_2$ 
rooted at $M/e$, $M \backslash e$, $M/f$, $M \backslash f$ respectively
to the leaves of $\mathcal{F}_0$
and $\mathcal{G}_0$.
 If $\Lambda_1$, $\Lambda_2$, $\Omega_1$, $\Omega_2$
 are the indecomposable Tutte coverings induced by 
$\mathcal{F}_1$, $\mathcal{F}_2$, $\mathcal{G}_1$, $\mathcal{G}_2$
respectively, then we evidently
 have the following equalities of multi-sets:

\begin{equation} \label{ind.partition} 
\mathrm{Ind}(\Lambda) = \mathrm{Ind}(\Lambda_1) \sqcup \mathrm{Ind}(\Lambda_2)
\end{equation}
\[
\mathrm{Ind}(\Omega) = \mathrm{Ind}(\Omega_1) \sqcup \mathrm{Ind}(\Omega_2).
\]

\vspace{0.2cm} 

We shall break down the rest of the inductive step into two cases.   \\ 

\textit{Case 1: $e$ and $f$ are either parallel or coparallel elements of $M$.} In this case, we have isomorphisms of the
form $M/e \cong M/f$ and $M\backslash e \cong M \backslash f$. These matroid isomorphisms induce 
in turn isomorphisms of multi-sets of the form
\[
\mathrm{Ind}(\Lambda_1) \cong  \mathrm{Ind}(\Omega_1) \qquad \qquad \mathrm{Ind}(\Lambda_2) \cong \mathrm{Ind}(\Omega_2),
\]
which then evidently imply that $\mathrm{Ind}(\Lambda) \cong  \mathrm{Ind}(\Omega)$.  \\ 

\textit{Case 2: $e$ and $f$ are not parallel or coparallel in $M$.} The assumption in this case implies that 
$f$ (resp. $e$) is a non-degenerate element in both $M/e$ and $M\backslash e$ 
(resp. $M/f$ and $M\backslash f$). This fact ensures that we can construct elementary 
deletion-contraction trees of the form 
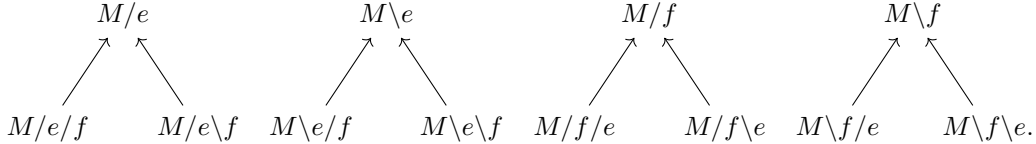
\[
\begin{tikzpicture}
      \tikzset{enclosed/.style={draw, circle, inner sep=0pt, minimum size=.1cm, fill=black}}
      
      \node (m/e) at (-6, 0) {$M/e$};
      \node (m/e/f) at (-7, -1.5) {$M/e/f$};
      \node (m/e_f) at (-5, -1.5) {$M/e \backslash f$};
      
      \draw (m/e/f) edge[->] (m/e);
      \draw (m/e_f) edge[->] (m/e);
      
      \node (m_e) at (-2.5, 0) {$M\backslash e$};
      \node (m_e/f) at (-3.5, -1.5) {$M\backslash e/f$};
      \node (m_e_f) at (-1.5, -1.5) {$M \backslash e \backslash f$};
      
      \draw (m_e/f) edge[->] (m_e);
      \draw (m_e_f) edge[->] (m_e);
      
      \node (m/f) at (1, 0) {$M/f$};
      \node (m/f/e) at (0, -1.5) {$M/f/e$};
      \node (m/f_e) at (2, -1.5) {$M/f \backslash e$};
      
      \draw (m/f/e) edge[->] (m/f);
      \draw (m/f_e) edge[->] (m/f);
      
      \node (m_f) at (4.5, 0) {$M\backslash f$};
      \node (m_f/e) at (3.5, -1.5) {$M\backslash f/e$};
      \node (m_f_e) at (5.5, -1.5) {$M \backslash f \backslash e$.};
      
      \draw (m_f/e) edge[->] (m_f);
      \draw (m_f_e) edge[->] (m_f);
       
\end{tikzpicture}
\]
Take now indecomposable Tutte coverings
\[
\Lambda_{11} \qquad \Lambda_{12} \qquad \Lambda_{21} \qquad \Lambda_{22} \qquad  
\Omega_{11} \qquad \Omega_{12} \qquad \Omega_{21} \qquad \Omega_{22}
\]
for the matroids 
\[
M/e/f \qquad M/e \backslash f \qquad M\backslash e /f \qquad M\backslash e \backslash f \qquad
M/f/e \qquad M/f \backslash e \qquad M\backslash f /e \qquad M\backslash f \backslash e   
\]
respectively. It follows from the induction hypothesis
that we have isomorphisms of the form 
\begin{equation} \label{partition.isos}
\mathrm{Ind}(\Lambda_1) \sqcup \mathrm{Ind}(\Lambda_2)  \cong \mathrm{Ind}(\Lambda_{11}) \sqcup \mathrm{Ind}(\Lambda_{12}) \sqcup
\mathrm{Ind}(\Lambda_{21}) \sqcup \mathrm{Ind}(\Lambda_{22})
\end{equation}
\[
\mathrm{Ind}(\Omega_1) \sqcup \mathrm{Ind}(\Omega_2) \cong \mathrm{Ind}(\Omega_{11}) \sqcup \mathrm{Ind}(\Omega_{12}) \sqcup
\mathrm{Ind}(\Omega_{21}) \sqcup \mathrm{Ind}(\Omega_{22}).
\]
\vspace{0.2cm}
On the other hand, the identities  
\[
M/e/f = M/f/e \qquad M/e \backslash f = M\backslash f / e \qquad M\backslash e / f = M /f \backslash e 
\qquad M\backslash f \backslash e = M\backslash e \backslash f
\]
ensure that we also have isomorphisms 
\begin{equation} \label{parts.isos}
\mathrm{Ind}(\Lambda_{11}) \cong \mathrm{Ind}(\Omega_{11}) \qquad  \mathrm{Ind}(\Lambda_{12}) \cong \mathrm{Ind}(\Omega_{21}) \qquad
\mathrm{Ind}(\Lambda_{21}) \cong \mathrm{Ind}(\Omega_{12}) \qquad \mathrm{Ind}(\Lambda_{22}) \cong \mathrm{Ind}(\Omega_{22}).
\end{equation}
Then, by combining (\ref{ind.partition}), (\ref{partition.isos}), and (\ref{parts.isos}), we can conclude that 
$\mathrm{Ind}(\Lambda) \cong \mathrm{Ind}(\Omega)$, which is precisely what we wanted to prove. 
\end{proof}

We can now provide the final details of the proof of Theorem \ref{thma}.

\begin{proof}[Proof of Theorem \ref{thma}] 
From Theorem \ref{k0.result}, Proposition \ref{ind.refine}, and Proposition \ref{ind.ind}, it follows that 
$K_0(\mathbf{Mat}^{\mathrm{tc}})$ is the free abelian group generated by the set 
\[
\Big\{ \hspace{0.05cm} [\varepsilon^m \oplus \sigma^n] \hspace{0.15cm} | \hspace{0.15cm} m,n \geq 0  \hspace{0.05cm} \Big\}
\]
of isomorphism classes of indecomposable objects.  
Thus, the assignment $[\varepsilon^m \oplus \sigma^n] \mapsto x^my^n$
defines an isomorphism
\[
\rho: K_0(\mathbf{Mat}^{\mathrm{tc}}) \stackrel{\cong}{\longrightarrow} \mathbb{Z}[x,y] 
\]
of abelian groups. This concludes the proof of part (i) of Theorem \ref{thma}. 

For the second claim, consider the following diagram of abelian groups:
\begin{equation} \label{diag.thma}
\quad
\end{equation}
\vspace{-1cm}
\[
\begin{tikzpicture}
      \tikzset{enclosed/.style={draw, circle, inner sep=0pt, minimum size=.1cm, fill=black}}
      
      \node (k0i) at (0,0) {$K_0(\mathbf{Mat}^{\cong})$};
       \node (k0tc) at (3.5,0) {$K_0(\mathbf{Mat}^{\mathrm{tc}})$};
      \node (ZM) at (0,-1.75) {$\mathbb{Z}[\mathcal{M}]$};
      \node (Zxy) at (3.5,-1.75) {$\mathbb{Z}[x,y]$.};
      
      \draw (k0i) edge[->] node[midway, above] {$\gamma$}   (k0tc);
      \draw (k0i) edge[->] node[midway, left] {$=$}  (ZM);
      \draw (ZM) edge[->] node[midway, above] {$\textcolor{white}{a}_{\mathcal{T}}$}  (Zxy);
      \draw (k0tc) edge[->]  node[midway, right] {$\rho$} (Zxy) ;
      
      \end{tikzpicture}
\]    
The bottom map $\mathcal{T}$ is the \textit{Tutte polynomial map}, i.e.,
it is the group homomorphism which sends a generator $[M]$ to the
Tutte polynomial $T(M;x,y)$.  
We wish show that the diagram in (\ref{diag.thma}) commutes. In
other words, we must show that the map $\mathcal{T}$ is equal to the composition 
$\rho\circ \gamma$, which we will denote simply as $\widetilde{\mathcal{T}}$. 
To do this, we shall prove that the assignment
$[M] \mapsto \widetilde{\mathcal{T}}([M])$ has the following properties: 

\begin{enumerate}

\item $\widetilde{\mathcal{T}}([\varepsilon^m\oplus \sigma^n]) = x^my^n$ for any non-negative integers $m$ and $n$.  
\vspace{0.15cm}

\item If $e$ is a non-degenerate element of $M$, then $\widetilde{\mathcal{T}}([M]) = \widetilde{\mathcal{T}}([M / e]) + \widetilde{\mathcal{T}}([M \backslash e])$. 

\end{enumerate}

Once we have verified that the assignment 
$[M] \mapsto \widetilde{\mathcal{T}}([M])$ satisfies the two properties given above,
an induction argument identical to the one we did in the proof of Proposition \ref{ind.refine}
implies that 
\begin{equation} \label{tutte.equality}
\widetilde{\mathcal{T}}([M]) = \mathcal{T}([M])
\end{equation}
for any matroid $M$, i.e., we first show that
(\ref{tutte.equality}) holds for isomorphism classes of the form 
$[\varepsilon^m\oplus \sigma^n]$ and then extend this result 
inductively to all matroids by contracting and deleting non-degenerate elements.  
Thus, once we verify properties (1) and (2) above for $\widetilde{\mathcal{T}}$, 
it will automatically follow that the diagram in (\ref{diag.thma}) commutes. 

To show that $\widetilde{\mathcal{T}}$ satisfies the desired properties, it is convenient
to have different notation for generators in $K_0(\mathbf{Mat}^{\cong})$ and $K_0(\mathbf{Mat}^{\mathrm{tc}})$: 
We shall continue to denote the generator in $K_0(\mathbf{Mat}^{\cong})$ corresponding to a matroid 
$M$ by $[M]$; on the other hand, in $K_0(\mathbf{Mat}^{\mathrm{tc}})$, we shall denote the generator
corresponding to $M$ by $\langle M \rangle$.  
By the definition of the maps $\gamma: K_0(\mathbf{Mat}^{\cong}) \rightarrow K_0(\mathbf{Mat}^{\mathrm{tc}})$
and $\rho: K_0(\mathbf{Mat}^{\mathrm{tc}}) \rightarrow \mathbb{Z}[x,y]$, we have that
\[
\widetilde{\mathcal{T}}([\varepsilon^m\oplus \sigma^n]) = \rho\circ \gamma(([\varepsilon^m\oplus \sigma^n])) = \rho(\langle \varepsilon^m\oplus \sigma^n \rangle) = x^my^n,
\]
which shows that the assignment $[M] \mapsto \widetilde{\mathcal{T}}([M])$ satisfies property (1). 
On the other hand, if $e$ is a non-degenerate element of a matroid $M$, then the standard inclusions 
$M/e \hookrightarrow M$ and $M\backslash e \hookrightarrow M$ form a Tutte covering for $M$. It follows that
\[
\rho\big(\gamma([M]\big) = \rho(\langle M \rangle) = \rho(\langle M/e \rangle) + \rho(\langle M\backslash e \rangle) = 
\rho\big(\gamma([M/e]) \big) + \rho\big(\gamma([M\backslash e]) \big).
\]
Therefore, the assignment $[M] \mapsto \widetilde{\mathcal{T}}([M])$ also satisfies property (2). 
By our previous discussion, we can now conclude that diagram (\ref{diag.thma}) is commutative.  
\end{proof}

\subsection{Ring structures} \label{ring.struct.sec}

For the proof of Theorem \ref{thmb}, we also need to show that the covering family structure 
on $\mathbf{Mat}^{\mathrm{tc}}$ is distributive with respect to direct sums, as indicated in the next proposition.

\theoremstyle{plain}  \newtheorem{tutcov.oplus}[iso.cov.str]{Proposition}

\begin{tutcov.oplus} \label{tutcov.oplus}
Consider two matroids $M, M' \neq \ast$ in $\mathbf{Mat}_{+}$.
If $\Lambda = \{f_i: N_i \rightarrow M\}_{i \in I}$ and 
$\Omega = \{g_j: N'_j \rightarrow M'\}_{j \in J}$ are Tutte coverings 
for $M$ and $M'$ respectively, then the multi-morphism 
\begin{equation} \label{oplus.map}
\{ f_i\oplus g_j : N_i \oplus N'_j \rightarrow M \oplus M' \}_{(i,j)\in I\times J}
\end{equation}
is a Tutte covering for $M \oplus M'$. Moreover, if 
$\Lambda$ and $\Omega$ are indecomposable 
Tutte coverings, then so is the Tutte covering in (\ref{oplus.map}).  
\end{tutcov.oplus}

\begin{proof}
The morphism $f_i \oplus g_j$ appearing in the multi-set (\ref{oplus.map})
is the morphism whose underlying set function is the coproduct $E_{N_i} \sqcup E_{N'_j}\rightarrow E_{M} \sqcup E_{N}$
of the set functions $E_{N_i} \rightarrow E_{M}$ and $E_{N'_j} \rightarrow E_{N}$ underlying $f_i$ and $g_j$ respectively.
To prove this proposition, we must show that (\ref{oplus.map})
can be realized as the collection of leaf-to-root morphisms 
of some deletion-contraction tree $\mathcal{F}$. 
Fix then deletion-contraction trees $\mathcal{F}_1$ and $\mathcal{F}_2$
so that $\{f_i: N_i \rightarrow M\}_{i \in I}$ and 
$\{g_j: N'_j \rightarrow M'\}_{j \in J}$ are the leaf-to-root morphisms of
$\mathcal{F}_1$ and $\mathcal{F}_2$ respectively. 
From now on, we will assume that the indexing set $I$ is of the form
$I = \{ 1, \ldots, p\}$.
Also, we will denote the identity morphisms 
$\mathrm{Id}_{M'}$, $\mathrm{Id}_{N_1}$, \ldots \ldots, $\mathrm{Id}_{N_p}$
by $h'$, $h_1$, \ldots, $h_p$ respectively.  
Note that we can produce a deletion-contraction tree 
$\mathcal{F}'_1$ rooted at $M\oplus M'$ by applying the 
functor $-\oplus M'$ to the tree $\mathcal{F}_1$. Then, 
the multi-set of  leaf-to-root morphisms of $\mathcal{F}'_1$
is equal to 
\[
\big\{f_i\oplus h': N_i \oplus M' \rightarrow M\oplus M'\big\}_{i \in \{1, \ldots, p\} }.
\]  
Similarly, for each $i \in I =  \{ 1, \ldots, p\}$, we can produce a 
deletion-contraction tree $\mathcal{F}'_{2i}$ 
rooted at $N_i \oplus M'$ by applying the functor 
$N_i \oplus -$ to $\mathcal{F}_2$.   
Evidently, for $i$ in $\{ 1, \ldots, p\}$, the multi-set of 
leaf-to-root morphisms of $\mathcal{F}'_{2i}$ is 
\[
\big\{h_i\oplus g_j: N_i \oplus N'_j \rightarrow N_i \oplus M'\big\}_{j \in J}.
\]  
We can then obtain the desired deletion-contraction tree $\mathcal{F}$ 
by attaching $\mathcal{F}'_{21}$, \ldots, $\mathcal{F}'_{2p}$
to the tree $\mathcal{F}'_1$ at the leaves 
$N_1 \oplus M'$, \ldots, $N_p \oplus M'$ respectively. 
By construction, the multi-set of leaf-to-root morphisms of 
$\mathcal{F}$ matches the one given in (\ref{oplus.map}),
which concludes the proof of the first claim.
The second claim follows from the fact that
the direct sum of two indecomposable objects 
is again indecomposable.  
\end{proof}

The morphisms $\gamma, \rho,$ and $\mathcal{T}$ appearing in 
the statement of Theorem \ref{thma} are for now just homomorphisms 
of abelian groups. However, as noted in the introduction,
we can use the direct sum operation to endow both $K_0(\mathbf{Mat}^{\cong})$
and $K_0(\mathbf{Mat}^{\mathrm{tc}})$ with ring structures.
With these ring structures in place,   
the morphisms in Theorem \ref{thma} become  ring homomorphisms, as claimed
in Theorem \ref{thmb}.
We will provide the details of these constructions in the following proof.  

\begin{proof}[Proof of Theorem \ref{thmb}]
Let us start by defining a product on $\mathbb{Z}[\mathcal{M}]$ (which, as remarked several times before, 
is equal to $K_0(\mathbf{Mat}^{\cong})$). We define a product on the generators of 
$\mathbb{Z}[\mathcal{M}]$ via the rule
\begin{equation} \label{ring.def}
[M]\cdot [N] := [M \oplus N]. 
\end{equation}
Evidently, if $M \cong M'$ and $N \cong N'$, then 
$M \oplus N \cong M' \oplus  N'$. Thus, the operation 
in (\ref{ring.def}) is well-defined. 
We then extend this product 
to arbitrary pairs in $\mathbb{Z}[\mathcal{M}]$ by setting 
\begin{equation} \label{ring.def.2}
\sum_{i} a_i[M_i] \cdot \sum_{j} b_j[N_j] := \sum_{i}\sum_{j} a_ib_j [M_i \oplus N_j].
\end{equation}
If $+$ is the addition operation on $\mathbb{Z}[\mathcal{M}]$, then it is 
straightforward to verify that the triple $(\mathbb{Z}[\mathcal{M}], +, \cdot)$ 
is a commutative ring. Furthermore, the multiplicative unit 
of this ring is the class $[\varnothing]$ corresponding to the empty matroid. 

Now, consider the subgroup $H$ of  $\mathbb{Z}[\mathcal{M}]$
generated by elements of the form
\begin{equation} \label{gen.thmb}
[N] -\big( [N_1] + \ldots\ldots + [N_p] \big),
\end{equation}
where $N_1$, \ldots, $N_p$ are the domains 
of a Tutte covering for $N$.
As explained in
Remark \ref{cov.mat.morhp}, the group 
$K_0(\mathbf{Mat}^{\mathrm{tc}})$ is equal to the 
quotient $\mathbb{Z}[\mathcal{M}]/H$, and the morphism 
$\gamma: K_0(\mathbf{Mat}^{\cong}) \rightarrow K_0(\mathbf{Mat}^{\mathrm{tc}})$
between $K_0$ groups is equal to
the canonical quotient map 
$\mathbb{Z}[\mathcal{M}] \rightarrow \mathbb{Z}[\mathcal{M}]/H$. 
To prove Theorem \ref{thmb},
we will first show that the product defined in (\ref{ring.def.2}) descends
to a product in  $\mathbb{Z}[\mathcal{M}]/H$.  
To accomplish this,
it is enough to show that $H$ is an ideal of the ring 
$(\mathbb{Z}[\mathcal{M}], +, \cdot)$. 
Moreover, to verify that $H$ is an ideal, it is enough to show
that the product of a generator $[M]$ of $\mathbb{Z}[\mathcal{M}]$
and a generator of $H$ of the form (\ref{gen.thmb}) lies in $H$. 
Fix then an arbitrary matroid $M$ and a Tutte covering 
$\{ g_j: N_j \rightarrow N\}_{j=1,\ldots,p}$ for a matroid $N$,
and consider the product 
\begin{equation} \label{prod.thmb}
[M]\cdot \Big( [N] -\big( [N_1] + \ldots\ldots + [N_p] \big)\Big).
\end{equation}
By the definition of the product $\cdot$ on $\mathbb{Z}[\mathcal{M}]$, we can express (\ref{prod.thmb}) as
\begin{equation} \label{prod2.thmb}
[M\oplus N] - \big(  [M\oplus N_1] + \ldots\ldots + [M \oplus N_p] \big).
\end{equation}
But, by Proposition \ref{tutcov.oplus}, the multi-set 
$\{ \mathrm{Id}_M \oplus g_j: M \oplus N_j \rightarrow M \oplus N\}_{j = 1,\ldots,p}$ is a Tutte covering
for $M\oplus N$. Consequently, the element in (\ref{prod2.thmb}) (or, equivalently,
the product in (\ref{prod.thmb})) lies in $H$. 
We can therefore conclude that
$H$ is an ideal of $\mathbb{Z}[\mathcal{M}]$ and that
the product $\cdot$ on $\mathbb{Z}[\mathcal{M}]$ descends 
to a product on $\mathbb{Z}[\mathcal{M}]/H$, 
which we shall also denote by $\cdot$.

Since the product $\cdot$
in  $\mathbb{Z}[\mathcal{M}]/H$ is induced by the one
in $\mathbb{Z}[\mathcal{M}]$, we evidently have that 
the group homomorphism 
$\gamma: \mathbb{Z}[\mathcal{M}] \rightarrow \mathbb{Z}[\mathcal{M}]/H$
preserves products, i.e., it is a ring homomorphism. 
Thus, to finish the proof of Theorem \ref{thmb},
we just need to show that the group isomorphism 
$\rho: \mathbb{Z}[\mathcal{M}]/H \rightarrow \mathbb{Z}[x,y]$
also preserves products.
To do this,
we shall from now on denote the classes 
$[\varepsilon]$ and $[\sigma]$ in 
 $\mathbb{Z}[\mathcal{M}]/H$
 (i.e., the classes represented by
 a single isthmus and a single loop respectively) simply by
 $\varepsilon$ and $\sigma$. 
 Also, any product of the form $\varepsilon^n\cdot \sigma^m$
 in  $\mathbb{Z}[\mathcal{M}]/H$ shall be written
 as $\varepsilon^n \sigma^m$. 
With this change in notation, 
we have that any 
element $\alpha$ of $\mathbb{Z}[\mathcal{M}]/H$ 
can be expressed uniquely as
$\alpha = \sum_{i=1}^{K} a_i\varepsilon^{n_i}\sigma^{m_i}$,
where $a_i \in \mathbb{Z}$. Also, by applying the distributive
property of the product $\cdot$, 
we have that the product $\alpha \cdot \alpha'$ of two elements 
$\alpha = \sum_{i=1}^{K} a_i\varepsilon^{n_i}\sigma^{m_i}$ and $\alpha' = \sum_{j=1}^{K'} b_j\varepsilon^{n_j'}\sigma^{m_j'}$
in $\mathbb{Z}[\mathcal{M}]/H$ is equal to 
\[
\alpha\cdot \alpha' = \sum_{i=1}^{K}\sum_{j = 1}^{K'} a_ib_j\varepsilon^{n_i + n'_j}\sigma^{m_i + m'_j}.  
\]
Therefore, as a ring, the quotient $\mathbb{Z}[\mathcal{M}]/H$ is equal to the
polynomial ring $\mathbb{Z}[\varepsilon, \sigma]$ generated by
 $\varepsilon$ and $\sigma$. Then, since
 the products in $\mathbb{Z}[\mathcal{M}]/H = \mathbb{Z}[\varepsilon, \sigma]$
 and $\mathbb{Z}[x, y]$ are identical once we replace 
 $\varepsilon$ and $\sigma$ with $x$ and $y$ respectively, 
 the map $\rho: \mathbb{Z}[\mathcal{M}]/H \rightarrow \mathbb{Z}[x, y]$
 also preserves products, which concludes the proof of this theorem.  
\end{proof}

\vspace{0.2cm}

The ring $\mathbb{Z}[\varepsilon, \sigma]$ introduced
in the last paragraph of the
proof of Theorem \ref{thmb} is known as the 
\textit{Tutte-Grothendieck ring}, defined by 
Brylawski in \cite{Br71}. As we did in the introduction, we shall denote this ring 
by $\mathcal{R}_{\mathrm{TG}}$.    

Now, consider again the free abelian group $\mathbb{Z}[\mathcal{M}]$
as a ring, where the product is the one we defined in the proof
of Theorem \ref{thmb} (i.e., the product defined via the rule
$[M]\cdot [N] := [M\oplus N]$).
As we showed in the proof of Theorem \ref{thmb}, 
the ring $K_0(\mathbf{Mat}^{\mathrm{tc}})$ is equal to
$\mathcal{R}_{\mathrm{TG}}$.
Then, as we previewed in the introduction, 
we can write the ring homomorphism  
$\gamma: K_0(\mathbf{Mat}^{\cong}) \rightarrow K_0(\mathbf{Mat}^{\mathrm{tc}})$ as  
\[
\gamma: \mathbb{Z}[\mathcal{M}] \longrightarrow \mathcal{R}_{\mathrm{TG}}.
\]
This map 
$\gamma$ is the ring homomorphism which
sends an isomorphism class $[M]$ to the polynomial
$T(M; \varepsilon, \sigma)$, i.e., the element in
$\mathcal{R}_{\mathrm{TG}} = \mathbb{Z}[\varepsilon, \sigma]$
obtained by evaluating the Tutte polynomial 
$T(M; x, y)$ on $\varepsilon$ and $\sigma$. 

In \cite{Br71}, Brylawski showed  that
the ring homomorphism 
$\gamma: \mathbb{Z}[\mathcal{M}] \rightarrow \mathcal{R}_{\mathrm{TG}}$
(called the \textit{Tutte polynomial} in \cite{Br71})
is actually the \textit{universal Tutte-Grothendieck invariant.}
More precisely, a \textit{Tutte-Grothendieck invariant} is a ring homomorphism 
$f: \mathbb{Z}[\mathcal{M}] \rightarrow R$ with the following properties 
(see Chapter \S 9 of \cite{GMc}):  

\begin{itemize}

\vspace{0.1cm} 

\item[$\cdot$] $f([M]) = f([M/e]) + f([M\backslash e])$ if $e$ is a non-degenerate element of $M$. 

\vspace{0.2cm} 

\item[$\cdot$] $f([M]) = f(\varepsilon)\cdot f([M/e])$ if $e$ is an isthmus of $M$. 

\vspace{0.2cm} 

\item[$\cdot$] $f([M]) = f(\sigma)\cdot f([M\backslash e])$ if $e$ is a loop of $M$.

\end{itemize}

\vspace{0.1cm}

Then, $\gamma: \mathbb{Z}[\mathcal{M}] \rightarrow \mathcal{R}_{\mathrm{TG}}$
is the universal Tutte-Grothendieck invariant
in the sense that,
given any Tutte-Grothendieck invariant $f: \mathbb{Z}[\mathcal{M}] \rightarrow R$, there
exists a unique ring homomorphism $g: \mathcal{R}_{\mathrm{TG}} \rightarrow R$
making the following diagram of rings commute: 
\[
\begin{tikzpicture}
      \tikzset{enclosed/.style={draw, circle, inner sep=0pt, minimum size=.1cm, fill=black}}

       \node (1) at (0,0) {$\mathbb{Z}[\mathcal{M}] $};
       \node (2) at (2.5,0) {$\mathcal{R}_{\mathrm{TG}}$};
       \node (3) at (2.5,-1.75) {$R$.}; 
      
      \draw (1) edge[->] node[midway,above] {\small $\gamma$} (2);
      \draw (1) edge[->] node[midway,below, yshift=-0.05cm, xshift = -0.15cm] {\small$f$} (3);
      \draw (2) edge[->] node[midway,right] {\small $g$} (3);
      
\end{tikzpicture}
\]   
It is a consequence of Theorem \ref{thma} that the universal Tutte-Grothendieck invariant 
$\gamma$ lifts (as a morphism of abelian groups) to a map of spectra. 
This is precisely the result given in Theorem \ref{thmc}, which we state 
again for the sake of completeness and presentation.    

\theoremstyle{plain}  \newtheorem{klift}[iso.cov.str]{Theorem}

\begin{klift} \label{klift}
The map of $K$-theory spectra 
\[
K(\Gamma): K(\mathbf{Mat}^{\cong}) \rightarrow  K(\mathbf{Mat}^{\mathrm{tc}})
\]
induced by the morphism $\Gamma: \mathbf{Mat}^{\cong} \rightarrow \mathbf{Mat}^{\mathrm{tc}}$ 
is a lift of the universal Tutte-Grothendieck invariant 
$\gamma: \mathbb{Z}[\mathcal{M}] \rightarrow \mathcal{R}_{\mathrm{TG}}$
(as a morphism of abelian groups) to the category of spectra.
\end{klift}

\theoremstyle{definition}  \newtheorem{klift.rem}[iso.cov.str]{Note}

\begin{klift.rem} \label{klift.rem}
In light of the previous theorem, it is natural to ask whether it is
possible to lift the universal Tutte-Grothendieck invariant as a 
ring homomorphism. We conjecture that this is indeed possible by 
adapting the methods of Zakharevich from \cite{Zak22} to the 
context of categories with covering families. 
In \cite{Zak22}, Zakharevich introduced the notion of symmetric 
monoidal assembler (Definition 7.10 in \cite{Zak22}),
and showed that the $K$-theory of a symmetric monoidal assembler 
$\mathcal{C}$ (which we shall denote by $K^{\mathrm{Asm}}(\mathcal{C})$)
is an $E_{\infty}$-ring spectrum, which implies that  
$K_0^{\mathrm{Asm}}(\mathcal{C})$ is a ring (Theorem 1.11 in \cite{Zak22}).
Additionally, a morphism $\mathcal{C} \rightarrow \mathcal{D}$ 
of symmetric monoidal assemblers induces a map 
$K^{\mathrm{Asm}}(\mathcal{C}) \rightarrow K^{\mathrm{Asm}}(\mathcal{D})$
of $E_{\infty}$-ring spectra, which in turn induces a ring homomorphism at
the $K_0$ level. 

In a forthcoming paper, by adapting the work of 
Zakharevich, we shall prove that the categories 
with covering families $\mathbf{Mat}^{\cong}$ 
and $\mathbf{Mat}^{\mathrm{tc}}$ are both
\textit{symmetric monoidal}, in a sense similar to 
Definition 7.10 in \cite{Zak22}. 
The monoidal structure on 
$\mathbf{Mat}^{\cong}$ and $\mathbf{Mat}^{\mathrm{tc}}$
is determined by the direct sum operation on matroids. 
Moreover, we shall also prove that the morphism 
$K(\mathbf{Mat}^{\cong}) \rightarrow K(\mathbf{Mat}^{\mathrm{tc}})$
is in fact a map of $E_{\infty}$-ring spectra, thus yielding a
spectrum-level lift of the universal Tutte-Grothendieck invariant 
as a ring homomorphism. 
\end{klift.rem}

\end{document}